\documentclass{article}

\usepackage{arxiv}

\usepackage[utf8]{inputenc} 
\usepackage[T1]{fontenc}    
\usepackage{hyperref}       
\usepackage{url}            
\usepackage{booktabs}       
\usepackage{amsfonts}       
\usepackage{nicefrac}       
\usepackage{microtype}      
\usepackage{lipsum}     
\usepackage{graphicx}
\usepackage{doi}
\usepackage{amsmath,amsthm}
\usepackage{algorithmic}
\usepackage{array}
\usepackage[caption=false,font=normalsize,labelfont=sf,textfont=sf]{subfig}
\usepackage{textcomp}
\usepackage{stfloats}
\usepackage{verbatim}
\usepackage{multicol}
\usepackage{algorithm}
\usepackage{color}
\def\BibTeX{{\rm B\kern-.05em{\sc i\kern-.025em b}\kern-.08em
    T\kern-.1667em\lower.7ex\hbox{E}\kern-.125emX}}
\usepackage{balance}
\usepackage{tikz}
\usepackage{tikz-3dplot}
\usetikzlibrary{calc}
\usetikzlibrary{arrows.meta}
\newcommand{\B}[1]{\mathbf{#1}}
\newcommand{\vecv}[1]{\mathbf{#1}}

\newcommand{\inc}{^\textup{inc}}
\newcommand{\scat}{^\textup{scat}} 
\newcommand{\tot}{^\textup{tot}}

\newcommand{\ii}{\mathrm{i}}
\newcommand{\e}{\mathrm{e}}
\newcommand{\vcurl}[1][]{\textup{\textbf{curl}}\ifthenelse{\equal{#1}{}}{}{_#1}}

\newcommand{\exterior}{^\textup{+}}

\newcommand{\SLO}{\vecv{S}_{\kappa}}
\newcommand{\DLO}{\vecv{C}_{\kappa}}
\newcommand{\SLOh}{\vecv{S}_{\kappa, h}}

\newcommand{\be}{\vecv{e}}

\newcommand{\norm}[2]{\left\|#1\right\|_{#2}}

\newcommand{\magnetic}{\gamma_{N}}
\newcommand{\tangential}{\gamma_{t}}

\newcommand{\half}{\frac{1}{2}}

\newcommand{\curl}{\textbf{curl}}

\newcommand{\sdiv}{\textbf{Div}_{\Gamma}}

\newcommand{\sgrad}{\textbf{Grad}_{\Gamma}}

\newcommand{\hsgrad}{\textbf{Grad}_{\Gamma_{h}}}
\newcommand{\scurl}{\text{curl}_{\Gamma}}
\newcommand{\hscurl}{\text{curl}_{\Gamma_{h}}}
\newcommand{\vscurl}{\textbf{curl}_{\Gamma}}

\newcommand{\tikzAngleOfLine}{\tikz@AngleOfLine}
  \def\tikz@AngleOfLine(#1)(#2)#3{%
  \pgfmathanglebetweenpoints{%
    \pgfpointanchor{#1}{center}}{%
    \pgfpointanchor{#2}{center}}
  \pgfmathsetmacro{#3}{\pgfmathresult}%
  }

\newtheorem{theorem}{Theorem}

\title{An OSRC Preconditioner for the EFIE}

\author{ 
    Ignacia Fierro-Piccardo \\
    Department of Mathematics\\
    University College London\\ 
    London WC1E 6BT, U.K.\\
    \texttt{ucahmib@ucl.ac.uk} \\
    \And
    Timo Betcke \\
    Department of Mathematics\\
    University College London\\ 
    London WC1E 6BT, U.K.\\
    \texttt{t.betcke@ucl.ac.uk} \\
}

\renewcommand{\shorttitle}{An OSRC Preconditioner for the EFIE}

\hypersetup{
pdftitle={An OSRC Preconditioner for the EFIE},
pdfauthor={Ignacia Fierro-Piccardo, Timo Betcke},
pdfkeywords={Preconditioner, OSRC approximation, Electric Field Integral Equation},
}

\begin{document}
\maketitle

\begin{abstract}
    The Electric Field Integral Equation (EFIE) is a well-established tool to solve electromagnetic scattering problems. However, the development of 
    efficient and easy to implement preconditioners remains an active research area. In recent years, operator preconditioning 
    approaches have become popular for the EFIE,  where the electric field boundary integral operator is regularised by multiplication with another 
    convenient operator. A particularly intriguing choice is the exact Magnetic-to-Electric (MtE) operator as regulariser. 
    But, evaluating this operator is as expensive as solving the original EFIE. In work by El Bouajaji, Antoine and 
    Geuzaine, approximate local Magnetic-to-Electric surface operators for the time-harmonic Maxwell equation were proposed. These
    can be efficiently evaluated through the solution of sparse problems.  This paper demonstrates the preconditioning 
    properties of these approximate MtE operators for the EFIE. The implementation is described and a number of numerical 
    comparisons against other preconditioning techniques for the EFIE are presented to demonstrate the effectiveness of this 
    new technique.
\end{abstract}


\section{Introduction}
The numerical simulation of time-harmonic waves scattered by perfect electric conductors (PECs) is of 
fundamental importance across the spectrum of electromagnetic applications.

Denote by $\be\inc$ an incident field. We are looking for the solution $\be\tot=\be\inc + \be\scat$ of the exterior scattering 
problem, that satisfies:

\label{eq:maxwell_system}
\begin{subequations}
 \begin{align}
  \vcurl\,\vcurl\,\be\tot-\kappa^2\be\tot &= 0&&\text{in }\Omega\exterior, \label{eq:maxwell_pde} \\
 \be\tot\times \pmb\nu&=0&&\text{on }\Gamma,\label{eq:maxwell_bnd}\\
\lim_{|\B{x}|\rightarrow\infty}|\B{x}|\left(\vcurl\,\be\scat\times\frac{\B{x}}{|\B{x}|}-\ii \kappa\be\scat\right) &= 0
.\label{eq:maxwell_silver_mueller}
\end{align}
 \label{eq:maxwell}
\end{subequations}

\noindent Here, $\kappa=\omega\sqrt{\epsilon_0\mu_0}$ denotes the
wavenumber of the problem, with $\omega$ denoting the frequency and
$\epsilon_0$ and $\mu_0$ the electric permittivity and magnetic
permeability in vacuum.  The PEC object is denoted by $\Omega^{-}\subset \mathbb{R}^{3}$ and it is enclosed by a smooth  
boundary $\Gamma = \delta \Omega^{-}$, also $\Omega^{+} = \mathbb{R}^{3}\setminus \overline{\Omega^-}$ denotes the propagation 
medium.  Frequently, the incident field is a plane wave
given by $\be\inc=\B{p}\e^{\ii \kappa \B{x}\cdot\B{d}}$, where $\B{p}$ is a
non-zero vector representing the polarisation of the wave, $\B{d}$
is a unit vector perpendicular to $\B{p}$ that gives the direction of
the plane wave and $\pmb\nu$ denotes the unit normal vector which is orthogonal
to the local tangent plane to the surface of the scatterer.

An integral equation formulation of this problem leads to an operational equation of the form

\begin{equation}\label{eq:efie}
\begin{split}
\SLO\vecv{u} = - \left(\frac{\vecv{I}}{2}+ \DLO\right)\vecv{f},
\end{split}
\end{equation}

with $\SLO$ the electric field integral operator, $\DLO$  the
magnetic field integral operator, $\vecv{f}$ the tangential trace of the incident data and $\vecv{u}$ the solution to the system (we will define all these quantities in Section \ref{sect:ops}). The above is a direct formulation, 
but one could equally choose an indirect formulation. See \cite{Buffa2003} for details.

For moderate mesh sizes, the discretisation of \eqref{eq:efie} can easily be solved
by LU decomposition. As the mesh width decreases, iterative solvers become necessary though, but are hampered by the
ill-conditioning of $\SLO$ after discretisation. A strategy to deal with this issue is to introduce a regularisation operator $\vecv{R}$ such that
the new operator system,
\begin{equation}\label{eq:prec_efie}
\begin{split}
\vecv{R}\SLO\vecv{u} = - \vecv{R}\left(\frac{\vecv{I}}{2}+ \DLO\right)\vecv{f},
\end{split}
\end{equation}
leads to well-conditioned discretisations. 

The most common example of $\vecv{R}$ is the Calderón Multiplicative Preconditioner \cite{andriulli2008multiplicative}. The drawback 
of this method, however, is the need to evaluate discrete operator products. To illustrate this, assume a
function $\phi$ in some Hilbert space, and operators $\vecv{A}$ and $\vecv{B}$ in
compatible Hilbert spaces. In order to evaluate the product
$\psi = \vecv{AB}\phi$
through Galerkin discretisations of the operators $\vecv{A}$ and $\vecv{B}$, we need
to compute a finite dimensional matrix product of the form
$\psi_{h} =
\vecv{A}_{h}\vecv{M}_{h}^{-1}\vecv{B}_{h}\phi_{h},
$
where we have used the subscript $_{h}$ to denote finite dimensional quantities after
discretisation \cite{Betcke20}. The matrix $\vecv{M}_{h}$ is a mass matrix, which
contains the inner product of the test space of $\vecv{B}_{h}$ and the domain
space of $\vecv{A}_{h}$. The difficulty is that this mass matrix is numerically
singular for the standard choice of Rao-Wilton-Glisson (RWG) basis functions in electromagnetic
scattering \cite{buffa2007dual}. In order to overcome this problem, one can
use the so-called Buffa-Christiansen (BC) bases as the range space of the EFIE operator \cite{andriulli2008multiplicative}. 
However, their use is expensive, since the construction of BC functions requires barycentric
mesh refinements (see Fig. \ref{fig:meshes} for reference).

\begin{figure}[h!]
    \begin{center}
    \begin{tikzpicture}[scale=2]
        \path (0,0) coordinate (origin);
        \path (0:1cm) coordinate (P0);
        \path (1*72:1cm) coordinate (P1);
        \path (2*72:1cm) coordinate (P2);
        \path (3*72:1cm) coordinate (P3);
        \path (4*72:1cm) coordinate (P4);
      
        \draw (P0) -- (P1) -- (P2) -- (P3) -- (P4) --cycle;
        \draw (origin) -- (P0)  (origin) -- (P1)
              (origin) -- (P2)  (origin) -- (P3)
              (origin) -- (P4);
      \end{tikzpicture}
    \hspace{1cm}
    \begin{tikzpicture}[scale=2]
        \path (0,0) coordinate (origin);
        \path (0:1cm) coordinate (P0);
        \path (1*72:1cm) coordinate (P1);
        \path (2*72:1cm) coordinate (P2);
        \path (3*72:1cm) coordinate (P3);
        \path (4*72:1cm) coordinate (P4);
        \path (1*108:0.51cm) coordinate (P5);
         \path (1*108:0.81cm) coordinate (P6);
         \path (1*72:0.45cm) coordinate (P7);
          \path (2*72:0.45cm) coordinate (P8);
        \path (180:0.51cm) coordinate (P9);
         \path (180:0.81cm) coordinate (P10);
         \path (3*72:0.45cm) coordinate (P11);
         \path (252:0.51cm) coordinate (P12);
         \path (252:0.81cm) coordinate (P13);
         \path (4*72:0.45cm) coordinate (P14);
         \path (324:0.51cm) coordinate (P15);
         \path (324:0.81cm) coordinate (P16);
         \path (5*72:0.45cm) coordinate (P17);
         \path (36:0.51cm) coordinate (P18);
         \path (36:0.81cm) coordinate (P19);
         \path (6*72:0.45cm) coordinate (P20);
        \draw (P0) -- (P1) -- (P2) -- (P3) -- (P4) --cycle;
        \draw (origin) -- (P0)  (origin) -- (P1)
              (origin) -- (P2)  (origin) -- (P3)
              (origin) -- (P4);
          \draw[dotted]
              (origin)--(P5)
              (P2)--(P5)
              (P1)--(P5)
              (P5)--(P6)
              (P5)--(P7)
              (P5)--(P8)
              (origin)--(P9)
              (P2)--(P9)
              (P3)--(P9)
              (P9)--(P8)
              (P9)--(P10)
              (P9)--(P11)
              (origin)--(P12)
              (P12)--(P13)
              (P12)--(P14)
              (P12)--(P11)
              (P12)--(P3)
              (P12)--(P4)
              (origin)--(P15)
              (P15)--(P4)
              (P15)--(P0)
              (P15)--(P16)
              (P15)--(P17)
              (P15)--(P14)
              (origin)--(P18)
              (P18)--(P0)
              (P18)--(P1)
              (P18)--(P17)
              (P18)--(P19)
              (P18)--(P20);
      \end{tikzpicture}   
    \caption{Primal and barycentric (dashed lines) meshes.} 
    \label{fig:meshes}
\end{center}
    \end{figure}
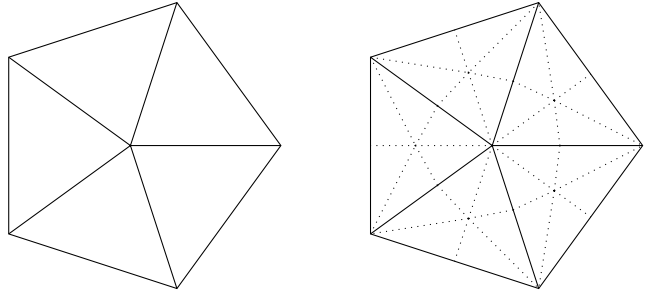

The thrust of this paper is to use the Magnetic-to-Electric (MtE) operator, while avoiding expensive barycentric mesh refinements. With this
choice, the left-hand side of \eqref{eq:prec_efie} becomes a second kind integral operator, which in theory, will make it more amenable 
to iterative solvers. However, a direct evaluation of the MtE operator has a similar complexity to solving the original scattering problem and is, therefore, impractical.
The idea of the On-Surface-Radiation-Condition (OSRC) approach \cite{kriegsmann1987new} is to obtain a high frequency approximation of the MtE operator
from a radiation condition applied on the surface of the scatterer.

In recent work by Bouajaji, Antoine, and Geuzaine \cite{el2014approximate}, the authors approximated the high-frequency symbol of the MtE operator by a Pad\'{e} expansion \cite{milinazzo1997rational} that
can be discretised using sparse surface operators and thereby, be efficiently evaluated. Earlier work in the 2D electromagnetic and acoustic cases is included in \cite{antoine2006improved,antoine2007generalized,darbas13}.

In this research we follow the aforementioned approach for approximating the MtE operator and investigate
its practical suitability as a regulariser for the EFIE.  We also demonstrate details of the
implementation and the resulting preconditioning performance, with
respect to assembly and solving times.

This paper is organised as follows: in section \ref{sect:PS} an overview of function spaces for Maxwell's boundary integral 
operators is presented. Section \ref{sect:MtEP} shows the MtE preconditioner, its continuous implementation,
approximations, its discrete implementation and simplifications connsidered.  
Section \ref{sect:NExp} presents a validation of this preconditioner and performance benchmarks, with concluding remarks
in Section \ref{sect:CR}.
  
\section{Problem setting, Tangential Sobolev Spaces and Surface Operators}\label{sect:PS}

We will start by introducing some required surface differential operators for the later sections. We will then discuss the function space setting and the discrete representation. Finally, we will introduce the required boundary operators.

\subsection{Surface differential operators}

In this section we briefly introduce the required surface differential operators. For a more complete technical definition see e.g. \cite{nedelec}. For the sake of the definitions here we assume a smooth bounded domain $\Omega$ with sufficiently smooth boundary $\Gamma$.

Let $u$ be a sufficiently smooth scalar vector field defined on $\Gamma$ and $\vecv{v}$ a sufficiently smooth tangential vector field defined on $\Gamma$, that is $\vecv{v}\cdot\pmb{\nu} = 0$. Furthermore, denote by $\tilde{u}$ and $\tilde{\vecv{v}}$ suitable extensions into 
a neighbourhood $\Gamma_{\epsilon}$ of $\Gamma$ with $\vecv{\tilde{v}}$ also requiring that it is tangential to surfaces in $\Gamma_{\epsilon}$ parallel to $\Gamma$. We define the following operators.

\begin{itemize}
\item The surface gradient
$$
\vecv{Grad}_{\Gamma}:= \nabla\tilde{u}|_{\Gamma}.
$$
\item The tangential curl
$$
\vecv{curl}_{\Gamma} := \vecv{Grad}_{\Gamma}u\times \pmb{\nu}.
$$
\item The surface divergence
$$
\vecv{Div}_{\Gamma}\;\vecv{v} :=(\text{div}\;\tilde{\vecv{v}})|_{\Gamma}.
$$
\item The surface curl
$$
\text{curl}_{\Gamma}\; \vecv{v}:= \pmb{\nu}\cdot (\text{curl}\; \tilde{\vecv{v}})|_{\Gamma}.
$$
\end{itemize}

We have the following identities (see \cite[Theorem 2.5.19]{nedelec})

$$
\begin{aligned}
\int_{\Gamma}\vecv{Grad}_{\Gamma}u\cdot \vecv{v}\;d\Gamma &= -\int_{\Gamma}u\,\vecv{Div}_{\Gamma}\,\vecv{v}\;d\Gamma,\\
\int_{\Gamma}\left(\vecv{curl}_{\Gamma}\,u\cdot \vecv{v}\right)d\Gamma &= \int_{\Gamma}u\,\text{curl}_{\Gamma}\vecv{v}\,d\Gamma,\\
\vecv{Div}_{\Gamma}\,\vecv{curl}_{\Gamma}\,u &= 0,\\
\text{curl}_{\Gamma}\vecv{Grad}_{\Gamma}u &= 0,\\
\vecv{Div}_{\Gamma}\,(\vecv{v}\times \pmb{\nu}) &= \text{curl}_{\Gamma}\,\vecv{v}.
\end{aligned}
$$

Moreover, we can define the scalar surface Laplace operator as
$$
\Delta_{\Gamma} u:= \vecv{Div}_{\Gamma}\;\vecv{Grad}_{\Gamma} u = -\text{curl}_{\Gamma}\;\vecv{curl}_{\Gamma} u,
$$
and the vectorial surface Laplace operator as
$$
\Delta_{\Gamma}\vecv{v} := \vecv{Grad}_{\Gamma}\;\vecv{Div}_{\Gamma}\vecv{v} - \vecv{curl}_{\Gamma}\;\text{curl}_{\Gamma}\vecv{v}
$$
(see \cite[2.5.191 \& 2.5.192]{nedelec}).

\subsection{Function Spaces}

Consider sufficiently smooth vector fields $\vecv{u}$ and $\vecv{v}$ such that the following implication of Green's formula makes sense:

\begin{equation}\label{eq:dual_prod}
\displaystyle \int_{\Omega}(\vecv{u}\cdot \curl\; \vecv{v}- \vecv{v}\cdot \curl\; \vecv{u}) \; d\Omega= \int_{\Gamma} \tangential \vecv{u} \cdot \vecv{v}|_{\Gamma} \; d\Gamma.
\end{equation}

The operator $\tangential\vecv{u}$ is the tangential trace: the product $\vecv{u}\times \pmb{\nu}$ taken on the boundary $\Gamma$.

Now define the tangential component trace $\pi_{t}$ as 

$$
\pi_t\vecv{v} := \vecv{v} - (\pmb{\nu}\cdot{\vecv{v}})\pmb{\nu} = \pmb{\nu}\times(\tangential \vecv{v})
$$

on $\Gamma$.

Equation \eqref{eq:dual_prod} introduces a duality relationship between $\tangential$ and $\pi_t$, since

$$
 \int_{\Gamma} \tangential \vecv{u} \cdot \vecv{v}|_{\Gamma} \; d\Gamma =  \int_{\Gamma} \tangential \vecv{u} \cdot \pi_t\vecv{v} \; d\Gamma.
$$

Moreover, the same formula motivates a self-duality for tangential traces through
$$
\int_{\Gamma} \tangential \vecv{u} \cdot \pi_t\vecv{v} \; d\Gamma = \int_{\Gamma}\tangential\vecv{u}\cdot(\pmb{\nu}\times \tangential\vecv{v})\; d\Gamma=:\langle\tangential\vecv{u}, \tangential\vecv{v}\rangle_{\times}.
$$

From \eqref{eq:dual_prod} it follows that the dual form $\langle\cdot, \cdot\rangle_{\times}$ makes sense for tangential traces of functions whose curl is well defined.

In the early 2000s the underlying ideas were made precise in the context of Sobolev spaces on bounded Lipschitz domains (see \cite{BuCa01, buffa2002traces}). A beautiful summary is also given in the overview paper \cite{Buffa2003}.

Here, we just summarise the key result with respect to the trace of $\tangential$ and the dual form $\langle\cdot, \cdot\rangle_\times$.

Let $\Omega$ be bounded (the case of an unbounded domain is similar). Define
$$
\vecv{H}^{s}(\vecv{curl}, \Omega) := \{\vecv{u}\in \vecv{H}^s(\Omega)|\;\vecv{curl}\;\vecv{u}\in \vecv{H}^s(\Omega)\}
$$
as the usual space for weak solutions of Maxwell's equations. In the following, we will denote $\mathbf{H} = \mathbf{H}^{0} $.
Then, one can define the space $\vecv{H}_{\times}^{-\frac{1}{2}}(\sdiv, \Gamma)$ of tangential traces on the boundary such that the mapping 
$\gamma_t:\vecv{H}(\vecv{curl}, \Omega)\rightarrow \vecv{H}_{\times}^{-\frac{1}{2}}(\sdiv, \Gamma)$ is continuous and surjective.
Moreover, this space is self-dual with respect to the dual form $\langle\cdot, \cdot\rangle_\times$ \cite{Buffa2003}. Correspondingly, 
we introduce the space $\vecv{H}_{\times}^{-\frac{1}{2}}(\scurl, \Gamma)$ as the continuous and surjective range of the map 
$\pi_t:\vecv{H}(\vecv{curl}, \Omega)\rightarrow \vecv{H}_{\times}^{-\frac{1}{2}}(\scurl, \Gamma)$. There is an isomorphism 
$\pmb{\Theta}:\vecv{H}_{\times}^{-\frac{1}{2}}(\sdiv, \Gamma)\rightarrow  \vecv{H}_{\times}^{-\frac{1}{2}}(\scurl, \Gamma)$ whose 
geometric interpretation on smooth boundaries is that $\pmb{\nu}\times \phi = \psi$ for some $\phi\in  \vecv{H}_{\times}^{-\frac{1}{2}}(\sdiv, \Gamma)$ 
and $\psi\in \vecv{H}_{\times}^{-\frac{1}{2}}(\scurl, \Gamma)$. The notations $\vecv{Div}_{\Gamma}$ and $\text{curl}_{\Gamma}$ in 
the names of the spaces make clear that they are weakly surface div and surface curl conforming, respectively. For precise definitions
see \cite{buffa2002traces}.

We will also require the usual scalar surface spaces $H^{1/2}(\Gamma)$ and $H^{-1/2}(\Gamma)$. The former can be interpreted as space of scalar Dirichlet data and the latter as space of scalar (weak) normal derivatives.
For all traces jumps and average operators can be defined as

\begin{align*}
  \left[\gamma\right]_{\Gamma} &= \gamma^{+}-\gamma^{-}\\
  \left\lbrace\gamma\right\rbrace_{\Gamma} &= \frac{\gamma^{+}+\gamma^{-}}{2}
\end{align*}

In order to build the discrete problem, we need to define discrete representations of  $\vecv{H}_{\times}^{-\frac{1}{2}}(\sdiv, \Gamma)$ and $\vecv{H}_{\times}^{-\frac{1}{2}}(\scurl, \Gamma)$.

Consider a polyhedral approximation $\Gamma_{h}$ of $\Gamma$ with a triangulation $\mathcal{T}_{h} = \cup_{l=1}^{N_{T}}T^{l}$,
we denominate as Raviart-Thomas (RT) the space of linear edge finite elements defined by the basis functions

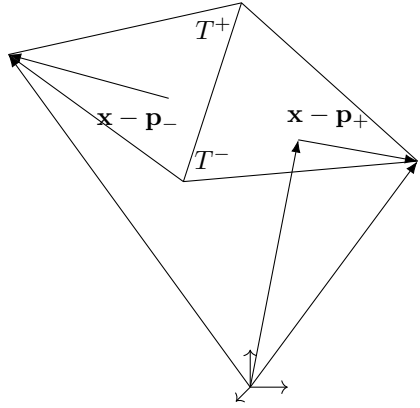
\begin{figure}[t]
    \begin{center}
    \begin{tikzpicture}[scale=2.5]
        \path (0,0) coordinate (origin);
        \path (5:1.25cm) coordinate (P0);
        \path (1*72:1cm) coordinate (P1);
        \path (2*72:1.15cm) coordinate (P2);
        \path (4*72:1.15cm) coordinate (P3);
        \path (20:0.65cm) coordinate (P4);
        \path (100:0.45cm) coordinate (P5);
      
        \draw (P0) -- (P1) -- (P2);
        \draw (origin) -- (P0)  (origin) -- (P1)
              (origin) -- (P2);
         \draw[-Latex](P4)--(P0);
         \draw[-Latex] (P3) -- (P0);
         \draw[-Latex](P3) -- (P2);
         \draw[-Latex](P3) -- (P4);  
         \draw[-Latex](P5) -- (P2);    
         \tikzAngleOfLine(origin)(P0){\AngleStart}
        \tikzAngleOfLine(origin)(P1){\AngleEnd}
        \tikzAngleOfLine(P2)(P1){\AngleEndd}
        \node[circle] at ($(origin)+({(\AngleStart+\AngleEnd)/2}:0.20cm)$) {$T^{-}$};
        \node[circle] at ($(P1)+({(\AngleEnd+\AngleEndd)/2}:-0.20cm)$) {$T^{+}$};
        \node[circle] at ($(P4)+({(\AngleEnd)/2}:0.20cm)$) {$\vecv{x} -\vecv{p} _{+}$};
        \node[circle] at ($(P5)+({(\AngleEnd)/2}:-0.20cm)$) {$\vecv{x} -\vecv{p} _{-}$};
        \draw[->] (P3)-- ($(P3)+(0.2,0,0)$) coordinate(X);
        \draw[->]  (P3) -- ($(P3)+(0,0.2,0)$) coordinate(Y);
        \draw[->]  (P3) -- ($(P3)+(0,0,0.2)$) coordinate(Z);
    
      \end{tikzpicture}
    \end{center}
    \caption{Definition of a RT basis function on an edge of the mesh} 
    \label{fig:rtbasis}
\end{figure}

$$
\vecv{RT}_{i}(\vecv{x})=  \begin{cases} \frac{1}{2A_{i}^{+}}(\vecv{x} -\vecv{p} _{+}) \;\; \text{ if } \vecv{x} \in  T_{+} \\ -\frac{1}{2A_{i}^{-}}(\vecv{x}-\vecv{p}_{-}) \text{ if } \vecv{x} \in T_{-} \\ \vecv{0} \text{, otherwise }  \end{cases}, \;\;i \in \left\lbrace 1, \dots, N_{e} \right\rbrace
$$

(see, for example, Fig. \ref{fig:rtbasis}), where $N_{e}$ is the number of edges in $\mathcal{T}_{h}$. 

We define Nédélec (NC) basis functions as the set of rotated RT basis functions:

$$\vecv{NC}_i:=\pmb\nu(\vecv{x})\times \vecv{RT}_{i}(\vecv{x}) $$

The RT and NC basis functions form discrete bases of the dual pair $\vecv{H}_{\times}^{-\frac{1}{2}}(\sdiv, \Gamma)$ and $\vecv{H}_{\times}^{-\frac{1}{2}}(\scurl, \Gamma)$.

We also define the Rao-Wilton-Glisson (RWG) basis functions (\cite{rao1982electromagnetic}) as a scaling of RT basis functions:

$$\vecv{RWG}_i:=l_{i} \vecv{RT}_{i}(\vecv{x}), $$

where $l_{i}$ is the length of the edge $i$ and the Scaled Nédélec (SNC) basis functions as the rotation of RWG:

$$\vecv{SNC}_i:=\pmb\nu(\vecv{x})\times \vecv{RWG}_{i}(\vecv{x})$$

Just like the pair RT-NC, RWG and SNC form discrete bases of the dual pair $\vecv{H}_{\times}^{-\frac{1}{2}}(\sdiv, \Gamma)$ and $\vecv{H}_{\times}^{-\frac{1}{2}}(\scurl, \Gamma)$.

Finally, we define piecewise linear basis functions (P1) on a reference element:

$$f_{i}(\xi, \eta) = \begin{cases}  1-\xi-\eta &\text{ for vertex 1}\\ \xi &\text{ for vertex 2}\\ \eta &\text{ for vertex 3}\end{cases}, \;\;i \in \left\lbrace 1, \dots, N_{v} \right\rbrace,$$

where $N_{v}$ is the number of vertices in $\mathcal{T}_{h}$. These are also called \emph{roof} basis functions; we intend to use them to discretise $H^{\frac{1}{2}}(\Gamma)$ and $H^{-\frac{1}{2}}(\Gamma)$ 
in the upcoming sections. 
%
%

In table \ref{tab:summary_basis} we summarise the notation used for the different sets of basis functions.

\begin{table}[h!]
\begin{center}
\begin{tabular}{|c|c|c|c|}
\hline
 Acronym & Type & Dofs & Discretises  \\ \hline
 RT    & Vectorial & Edges &$\vecv{H}_{\times}^{-\frac{1}{2}}(\sdiv, \Gamma)$ \\ \hline
 NC    & Vectorial & Edges &$\vecv{H}_{\times}^{-\frac{1}{2}}(\scurl,\Gamma)$  \\ \hline
 RWG   & Vectorial & Edges &$\vecv{H}_{\times}^{-\frac{1}{2}}(\sdiv, \Gamma)$  \\ \hline
 SNC   & Vectorial & Edges &$\vecv{H}_{\times}^{-\frac{1}{2}}(\scurl,\Gamma)$ \\ \hline
 P1& Scalar & Vertices &$H^{\frac{1}{2}}(\Gamma)$ and $H^{-\frac{1}{2}}(\Gamma)$  \\\hline
\end{tabular}
\end{center}
\caption{Summary of basis functions.}
\label{tab:summary_basis} 
\end{table}

\subsection{Operators}\label{sect:ops}

To solve (1a)-(1c), the Stratton-Chu representation formula \cite[Theorem 3.27]{kirsch2014mathematical} must be considered for any $\vecv{x}\in \Omega^{+}$

\begin{equation}\label{eq:rep_formula}
\be(\vecv{x}):= - \mathcal{T}(\magnetic^{+}\be)(\vecv{x})- \mathcal{K}(\tangential^{+}\be)(\vecv{x}).
\end{equation}
 
Here $\tangential^{+}\be$ in physical terms represents the surface magnetic current and consequently, $\magnetic^{+}\be$ is the surface electric current
\cite[equation (2.336)]{osipov2017modern}, defined by $\gamma_N^{+}\vecv{u}:=(i\kappa)^{-1}{{\tangential}^+ \curl\; \vecv{u}}$ (and similarly for the interior magnetic trace).
$\mathcal{T}, \mathcal{K}: \vecv{H}_{\times}^{-\frac{1}{2}}(\sdiv, \Gamma)\rightarrow \vecv{H}_{loc}(\curl^{2}, \Omega^{+})$ are defined as:

\begin{align*}
\mathcal{T}(\vecv{p})(\vecv{x})&:= i\kappa \int_{\Gamma} \vecv{p}(\vecv{y})\vecv{G}(\vecv{x}, \vecv{y})\\
& - \frac{1}{i\kappa}\nabla_{\vecv{x}}\int_{\Gamma}\vecv{G}(\vecv{x}, \vecv{y}) \sdiv\vecv{p}(\vecv{y}) d\Gamma(\vecv{y}),\\
\mathcal{K}(\vecv{p})(\vecv{x})&:= \curl_{\vecv{x}}\int_{\Gamma}\vecv{G}(\vecv{x}, \vecv{y}) \vecv{p}(\vecv{y}) d\Gamma(\vecv{y})
\end{align*}

With $\vecv{G}(\vecv{x}, \vecv{y}):= \frac{e^{i\kappa\norm{\vecv{x}-\vecv{y}}{}}}{4 \pi \norm{\vecv{x}-\vecv{y}}{}}$, $\vecv{x}\neq \vecv{y}$.\\

By applying magnetic and tangential traces to the Electric and Magnetic field potential operators, the electric and magnetic BIOs (Boundary Integral Operators)
$\SLO, \DLO: \vecv{H}_{\times}^{-\frac{1}{2}}(\sdiv, \Gamma)\rightarrow \vecv{H}_{\times}^{-\frac{1}{2}}(\sdiv, \Gamma)$ can be obtained:\\

\begin{equation*}
\begin{aligned}
\{\tangential\}_{\Gamma}\mathcal{T}&= \SLO & \left[\tangential\right]_{\Gamma}\mathcal{T}&= 0 \\
\{\tangential\}_{\Gamma}\mathcal{K}&= \DLO &  \left[\tangential\right]_{\Gamma}\mathcal{K}&= -\vecv{I} \\
\{\magnetic\}_{\Gamma}\mathcal{T}&= \DLO & \left[\magnetic\right]_{\Gamma}\mathcal{T}&= -\vecv{I}  \\
\{\magnetic\}_{\Gamma}\mathcal{K}&= -\SLO &\left[\magnetic\right]_{\Gamma}\mathcal{K}&= 0 
\end{aligned}
\end{equation*}

\begin{equation*}
\begin{aligned}
\tangential^{-}\mathcal{T}&= \SLO & \tangential^{+}\mathcal{T}&= \SLO\\
\tangential^{-}\mathcal{K}&= \frac{\vecv{I}}{2}+\DLO & \tangential^{+}\mathcal{K}&= -\frac{\vecv{I}}{2}+\DLO\\
\magnetic^{-}\mathcal{T}&= \frac{\vecv{I}}{2}+\DLO & \magnetic^{+}\mathcal{T}&= -\frac{\vecv{I}}{2}+\DLO\\
 \magnetic^{-}\mathcal{K}&= -\SLO & \magnetic^{+}\mathcal{K}&= -\SLO
\end{aligned}
\end{equation*}

Then, applying electric and magnetic traces to the representation formula  \eqref{eq:rep_formula}, the following can be derived:

\begin{equation}\label{eq:calderon_projector}
\mathcal{C}^{\pm} = \begin{bmatrix}
\frac{\vecv{I}}{2}\mp \DLO  & \mp \SLO\\
\pm \SLO& \frac{\vecv{I}}{2} \mp \DLO \\
\end{bmatrix} \begin{bmatrix}
\tangential^{\pm}\vecv{u}\\
\magnetic^{\pm}\vecv{u}\\
\end{bmatrix} 
 = \begin{bmatrix}
\tangential^{\pm}\vecv{u}\\
\magnetic^{\pm}\vecv{u}\\
\end{bmatrix}.
\end{equation}

The operator $\mathcal{C}^{\pm}$ is called Calder\'{o}n Projector. It describes the relationship of the electric and magnetic tangential traces on the boundary $\Gamma$.
An important property, directly following from the Stratton-Chu representation formula is that $\left(\mathcal{C}^{\pm}\right)^{2} = \mathcal{C}^{\pm}$, which implies
\begin{equation}\label{eq:CP}
\SLO^{2} =\DLO^{2}- \frac{\vecv{I}}{4}.
\end{equation}

This is the basis for Calder\'{o}n Preconditioning. It states that $\SLO^{2}$ is a compact perturbation of the identity on
sufficiently smooth domains, with eigenvalues clustering around the point $1/4$. Hence, under a suitable discretisation of this
operator, iterative solvers are expected to converge quickly. The difficulties of building the discrete version arise when trying to build Gram matrices $\vecv{G}$ to 
implement the discrete product $\SLOh \vecv{G}^{-1}\SLOh$. For standard RWG spaces, the matrix $\vecv{G}$ is singular \cite{andriulli2008multiplicative}. To overcome this
problem, in  \cite{andriulli2008multiplicative} a Calder\'{o}n multiplicative preconditioner (CMP) was proposed based on the use of BC basis functions that are defined
on barycentric refinements of the original grid (see \cite{buffa2007dual} for more details and Fig. \ref{fig:meshes} for reference). While the improvement in iterative 
solver convergence with this preconditioner is excellent, the implementation requires the assembly of operators on
grids with six times as many elements as the original grid. Acceleration techniques such as the Fast Multipole Method (FMM) \cite{greengard1987fast} make this more manageable. 
Still, this is significantly more costly than the assembly on the original grid.

A recent approach to implement a Calderón Preconditioner is shown in \cite{adrian2019refinement}, where the authors aim to 
build a Multiplicative Calderón Preconditioner immune to the low frequency breakdown induced by the use of RWG basis functions. 
Here, the authors perform a quasi-Helmholtz decomposition of the EFIE in the static limit by using loop-star basis functions 
\cite{vecchi1999loop} obtained from linear combinations of RWG basis functions, thus avoiding BC basis functions.

The main drawback of this technique, however, is the need to solve a dense matrix system as part of the application of the 
preconditioner. The authors have pointed out that this can be achieved by preconditioning this dense matrix with specific methods, 
making it competitive with the original CMP. However, the implementation effort of this approach is substantial.
In this paper we demonstrate that OSRC based preconditioners achieve similar performance, without requiring barycentric refinements and with an implementation that only requires the solution of sparse linear systems that are straightforward to assemble.

\section{Construction of an OSRC Preconditioner}\label{sect:MtEP}

We start by reviewing the preconditioning properties the EtM (Electric-to-Magnetic) operator $\vecv{V}$ and its inverse: the MtE operator $\vecv{V}^{-1}$. Then, we review the Pad\'{e} approximation 
approach for these operators, obtained from \cite{el2014approximate}, and finally describe in detail how to use these operators as discrete preconditioners. We assume that both $\SLO$ and $\frac{1}{2}\vecv{I}+\DLO$ are invertible 
(the operators have no resonance at the wavenumber $\kappa$).

\subsection{OSRC operator as a preconditioner for the EFIE}

The EtM operator can be derived from the first row of \eqref{eq:calderon_projector}

\begin{align*}
 -\SLO^{-1}\left(\frac{\vecv{I}}{2}+\DLO \right)\tangential^{+}\vecv{u} = \magnetic^{+}\vecv{u}.
\end{align*}

Hence, the EtM and its inverse, the MtE operator, are given by:

\begin{multicols}{2}
  \begin{equation}\label{eq:etm1}
   \vecv{V}_{(1)} = -\SLO^{-1}\left(\frac{\vecv{I}}{2}+\DLO \right),
  \end{equation}\break
  \begin{equation}\label{eq:mte1}
    \vecv{V}_{(1)} ^{-1} = - \left(\frac{\vecv{I}}{2}+\DLO \right)^{-1} \SLO.
  \end{equation}
\end{multicols}

Alternatively, from the second row of \eqref{eq:calderon_projector}

\begin{align*}
\left(\frac{\vecv{I}}{2}+\DLO \right)^{-1}\SLO\tangential^{+}\vecv{u} = \magnetic^{+}\vecv{u}.
\end{align*}

A second version of the EtM exact operator and its inverse (MtE) can be obtained:

\begin{multicols}{2}
  \begin{equation}\label{eq:etm2}
   \vecv{V}_{(2)}  = \left(\frac{\vecv{I}}{2}+\DLO \right)^{-1}\SLO,
  \end{equation}\break
  \begin{equation}\label{eq:mte2}
    \vecv{V}_{(2)} ^{-1} = \SLO^{-1} \left(\frac{\vecv{I}}{2}+\DLO \right).
  \end{equation}
\end{multicols}

It is necessary to remark that the discrete versions of the pairs \eqref{eq:etm1}, \eqref{eq:mte2} and \eqref{eq:mte1}, \eqref{eq:etm2} are not necessarily the same,
since they might be defined on different discrete spaces. See \cite{scroggs2017software} for details.

To discriminate versions of the $\vecv{V}^{-1}$ operator, we have used the subscripts $_{(1)}$ and $_{(2)}$.

\begin{theorem}\label{theo:osrc_prec}

Assume that $\vecv{V}_{(1)}$ and $\vecv{V}_{(2)}$ and their inverses exist as defined above. It holds that
\begin{equation*}
\SLO\vecv{V}^{-1}_{(2)}  \equiv \left(\frac{\vecv{I}}{2}+\DLO\right)
\end{equation*}

and

\begin{equation*}
\vecv{V}^{-1}_{(1)}\SLO  \equiv \left(\frac{\vecv{I}}{2}-\DLO\right).
\end{equation*}

\end{theorem}

\begin{proof}

The first relationship follows from

\begin{equation*}
\SLO\vecv{V}^{-1}_{(2)} =\SLO\SLO^{-1} \left(\frac{\vecv{I}}{2}+\DLO \right) =  \left(\frac{\vecv{I}}{2}+\DLO \right) 
\end{equation*}

For the second identity we obtain 
\begin{equation*}
\vecv{V}_{(1)}^{-1}\SLO =- \left(\frac{\vecv{I}}{2}+\DLO \right)^{-1} \SLO \SLO.
\end{equation*}
To obtain the desired result we use \eqref{eq:CP} and replace $\SLO^{2} = -\left(\frac{\vecv{I}}{2}+\DLO\right)\left(\frac{\vecv{I}}{2}-\DLO\right)$.

\end{proof}
It follows that on sufficiently smooth domains applying $\vecv{V}_{(1)}^{-1}$ or $\vecv{V}_{(2)}^{-1}$ from the left/right to $\SLO$, respectively,
results in a compact perturbation of the identity.

From now on, we drop the subscripts, and we refer to $\vecv{V}_{(1)}^{-1}$ just as $\vecv{V}^{-1}$. 

We have seen that $\vecv{V}^{-1}$ is a good candidate for a preconditioner.
However, the construction of $\vecv{V}^{-1}$ is as expensive as solving the EFIE, so finding a good approximation is essential.

\subsection{Approximation of the MtE operator}

In \cite{el2014approximate}, it is shown that an approximation for the EtM on smooth surfaces is given by \cite{el2014approximate}

\begin{equation}\label{eq:appr2014}
\magnetic^{+}\vecv{u}\approx   -\pmb{\Lambda}_{1,\varepsilon}^{-1}\pmb{\Lambda}_{2,\varepsilon}(\pmb{\nu} \times \tangential^{+}\vecv{u}) = -\pmb{\Lambda}_{1,\varepsilon}^{-1}\pmb{\Lambda}_{2,\varepsilon}\pmb{\Theta}\tangential^{+}\vecv{u}\;\; \text{ on } \Gamma,
\end{equation}

where

\begin{equation*}
\pmb{\Lambda}_{2,\varepsilon} := \vecv{I} - \vscurl \frac{1}{\kappa^{2}_{\varepsilon}}\scurl,
\end{equation*}

and 

\begin{equation*}
\pmb{\Lambda}_{1,\varepsilon} := (\vecv{I}+\mathcal{J})^{1/2}, 
\end{equation*}

with

\begin{equation}\label{eq:J_def}
\mathcal{J} := \sgrad \frac{1}{\kappa^{2}_{\varepsilon}}\sdiv - \vscurl \frac{1}{\kappa^{2}_{\varepsilon}}\scurl.
\end{equation} 

Notice here that we have used the term $\kappa_{\varepsilon} = \kappa + i \varepsilon$ instead of $\kappa$. The term $\varepsilon>0$ is a damping parameter used by the 
authors in \cite{el2014approximate} to avoid singularities in the square root operator, whose optimal value is given by $\varepsilon = 0.39\kappa^{\frac{1}{3}}R^{-\frac{2}{3}}$ (R being the curvature radius of the surface).

Finally, we write the approximation of the EtM operator as  $\vecv{V}_{\varepsilon}:=  -\pmb{\Lambda}_{1,\varepsilon}^{-1}\pmb{\Lambda}_{2,\varepsilon}\pmb{\Theta}$ and 
by taking the inverse, the approximate MtE can also be found: $\vecv{V}^{-1}_{\varepsilon}:= -\pmb{\Theta}^{-1} \pmb{\Lambda}_{2,\varepsilon}^{-1}\pmb{\Lambda}_{1,\varepsilon}$.

\subsection{Approximation of $V_{\varepsilon}$ by surface differential operators}

The operator $\pmb{\Lambda}_{2,\varepsilon} $ can be discretised into a sparse matrix that 
can be readily inverted using sparse LU decomposition. However, $\pmb{\Lambda}_{1,\varepsilon}$ is a
pseudo-differential operator  whose calculation is more involved. In \cite{el2014approximate} a rotating branch cut Pad\'{e} approximation of the form

\begin{equation}\label{eq:pade_app}
(1+z)^{\half} \approx   R_{0}-\sum_{j=1}^{N_{p}}\frac{A_{j}}{B_{j}(1+B_{j}z)}
\end{equation}
is proposed. Let $\alpha=\frac{\pi}{2}$; the coefficients $A_{j}$ and $B_{j}$ are given by

\begin{align*}
A_{j} &= \frac{e^{-i\alpha/2}a_{j}}{[1+b_{j}(e^{-i\alpha}-1)]^2},\\
B_{j} &=\frac{b_{j}e^{-i\alpha}}{1+b_{j}(e^{-i\alpha}-1)},
\end{align*}
with $a_{j} = \frac{2}{2N_{p} +1}\sin^{2}\left(\frac{j\pi}{2 N_{p} +1} \right)$, $b_{j} = \cos^{2}\left(\frac{j\pi}{2 N_{p} +1} \right)$.

For $R_0$ we have
\begin{equation*}
    R_{0} = C_{0}+ \sum_{j=1}^{N_p}\frac{A_{j}}{B_{j}},
\end{equation*}
with
\begin{equation*}
    C_{0} = e^{i\alpha/2}\left(1+\sum_{j=1}^{N_p}\frac{a_{j}(e^{-i\alpha}-1)}{1+b_{j}(e^{-i\alpha}-1)}\right).
\end{equation*}

For more details on the derivation of these parameters we refer to \cite{milinazzo1997rational}.

Applying \eqref{eq:pade_app} to approximate $\pmb{\Lambda}_{1,\varepsilon} = (\vecv{I}+\mathcal{J})^{1/2}$, we obtain the operator 
\begin{equation*}
\tilde{\pmb{\Lambda}}_{1,\varepsilon} = \left(\vecv{I}R_{0}-\sum_{j=1}^{N_{p}}\frac{A_{j}}{B_{j}} (\vecv{I}+B_{j}\mathcal{J})^{-1}\right).
\end{equation*}
As a simplification, we introduce $\pmb{\Pi}_{j}:=\vecv{I}+B_{j}\mathcal{J}$,  and by substituting $\pmb{\Lambda}_{1,\varepsilon}$ with $\tilde{\pmb{\Lambda}}_{1,\varepsilon}$ in the MtE operator $\vecv{V}_{\varepsilon}^{-1}$, we obtain the Pad\'{e} approximate MtE operator:

\begin{equation}\label{eq:laplace_app}
\begin{split}
    \tilde{\vecv{V}}_{\varepsilon, N_p}^{-1}:=-\pmb{\Theta}^{-1}\pmb{\Lambda}_{2}^{-1}\left(\vecv{I}R_{0}-\sum_{j=1}^{N_{p}}\frac{A_{j}}{B_{j}} \pmb{\Pi}_{j}^{-1}\right).
\end{split}
\end{equation}

To implement $\tilde{\vecv{V}}_{\varepsilon, N_p}^{-1}$ as a preconditioner, we want to evaluate $\vecv{r} :=\tilde{\vecv{V}}_{\varepsilon, N_p}^{-1}\SLO \vecv{r}_{1}$ for some function $\vecv{r}_{1}\in\vecv{H}_{\times}^{-\frac{1}{2}}(\sdiv,\Gamma)$.

Let $\vecv{r}_{1,h}$ be the discrete approximation of $\vecv{r}_1$ using RWG basis functions. The main difficulty is the discrete evaluation of $\pmb{\phi}^{j} := \pmb{\Pi}_j^{-1}\SLO\vecv{r}_1$, of which the main part is the solution of

\begin{equation}\label{eq:pi_app}
\SLO\vecv{r}_{1} = \left[\vecv{I} + B_j \underbrace{\left(\sgrad \frac{1}{\kappa^{2}_{\varepsilon}}\sdiv - \vscurl \frac{1}{\kappa^{2}_{\varepsilon}}\scurl\right)}_{\mathcal{J}}\right]\pmb{\phi}^j.
\end{equation}

%

We follow \cite{el2014approximate} to solve this system discretely in the following way: let $z_h$ and $\rho_h^j$ be represented by P1 basis functions, and $\vecv{w}_h, \pmb{\phi}_h^j$ by SNC basis functions.

Then
\eqref{eq:pi_app} is equivalent (in the discrete weak sense) to solving
\begin{align}\label{eq:lambda1_wf}
\begin{cases}  
\int_{\Gamma_{h}} \pmb{\phi}_h^{j}\cdot \vecv{w}_h^{j} d\Gamma_{h} +B_{j}\left( \int_{\Gamma_{h}}\hsgrad \rho_h^{j} \cdot \vecv{w}_h^{j}d\Gamma_{h} \right.\\
\left.- \int_{\Gamma_{h}}\frac{1}{\kappa_{\varepsilon}^{2}} \hscurl\; \pmb\phi_h^{j} \cdot \hscurl\;\vecv{w}_h^{j}d\Gamma_{h} \right) = \int_{\Gamma_{h}} \SLO\vecv{r}_{1,h}\cdot \vecv{w}_h^{j} d\Gamma_{h},\\
\int_{\Gamma_{h}} \kappa_{\varepsilon}^{2}\rho_h^{j} z_h^{j} d\Gamma_{h} +\int_{\Gamma_{h}}\pmb\phi_h^{j} \cdot \hsgrad z_h^{j} d\Gamma_{h} = 0
\end{cases} 
\end{align}
for the discrete representation $\pmb{\phi}^j_h$ of $\pmb{\phi}^j$.

With $\pmb{\phi}_h^j$ computed, we can evaluate with $\vecv{v}_h$ given in a basis of SNC functions:

\begin{align}\label{eq:weak_pade_sum}
\int_{\Gamma_{h}}\vecv{r}_{2,h}\cdot \vecv{v}_h\;d\Gamma_{h} &= R_{0}\left[\int_{\Gamma_{h}}\SLO\vecv{r}_{1,h}\cdot\vecv{v}_h\;d\Gamma_{h}\nonumber\right.\\
&\left.-\frac{1}{R_{0}}\int_{\Gamma_{h}}\left(\sum_{j=1}^{N_{p}}\frac{A_{j}}{B_{j}} \pmb{\phi}_h^{j}\right)\cdot \vecv{v}_h\;d\Gamma_{h}\right].
\end{align}

We find $\vecv{r}_{3,h}$ solving discretely $\pmb{\Lambda}_{2,\varepsilon}\vecv{r}_{3,h} = \vecv{r}_{2,h}$ by
\begin{align}\label{eq:lambda2_wf}
 \int_{\Gamma_{h}} \vecv{r}_{3,h} \cdot \vecv{v}_hd\Gamma_{h}- \int \frac{1}{\kappa_{\varepsilon}^{2}}\hscurl \vecv{r}_{3,h}\cdot \hscurl \vecv{v}_h \;d\Gamma_{h} \nonumber \\
  = \int_{\Gamma_{h}}\vecv{r}_{2,h}\cdot\vecv{v}_h\; d\Gamma_{h}.
\end{align}

Finally, we obtain $\vecv{r}_h = -\pmb{\Theta}_{h}^{-1}\vecv{r}_{3,h}$, where $\pmb{\Theta}_{h}$ is the discrete version of the operator $\pmb{\Theta}$. 
In section \ref{sect:mat_rep}, we demonstrate that with a suitable choice of basis functions, the matrix representation of $\pmb{\Theta}_{h}$ is just the identity matrix.

\subsection{Matricial representation}\label{sect:mat_rep}

To illuminate the steps behind the discrete implementation, we write down the preconditioner in matrix form.  
Using the notation from \cite{el2014approximate}, we define the following matrices:

\begin{equation}
\begin{cases}
\mathbb{G} = \int_{\Gamma_{h}}\vecv{t}\cdot \vecv{r}\; d\Gamma_{h}, \;\; \mathbb{N}_{\varepsilon} = \int_{\Gamma_{h}} \frac{1}{\kappa_{\varepsilon, h}^{2}}\hscurl\vecv{t} \cdot\hscurl\vecv{r}\; d\Gamma_{h}\\
\mathbb{K}_{\varepsilon}= \int_{\Gamma_{h}}\kappa_{\varepsilon, h}^{2}\ell \; \lambda\; d\Gamma_{h},\;\; \mathbb{L} = \int_{\Gamma_{h}} \hsgrad \ell \cdot \vecv{t} \;d\Gamma_{h}, \\
\end{cases}
\end{equation}

where $\vecv{t}$ and $\vecv{r}$ are SNC basis functions and 
$\ell$, $\lambda$ are P1 basis functions. Notice that we have used $\kappa_{\varepsilon, h}$ to indicate that this parameter depends on the local curvature radius.

Now, we associate the coefficient vectors $\vec{\vecv{r}}_{*,h}$ with the functions $\vecv{r}_{*,h}$ (with similar notation for coefficient vectors of other functions of finite-dimensional bases).
By $\SLOh$ we denote the discrete matrix associated with the EFIE operator $\SLO$. 
The system \eqref{eq:lambda1_wf} corresponds to the matrix system

\begin{equation}\label{eq:matrix_system}
\begin{bmatrix}
(\mathbb{G}-B_{j}\mathbb{N}_{\varepsilon}) & B_{j}\mathbb{L} \\
\mathbb{L}^{T} & \mathbb{K}_{\varepsilon} \\
\end{bmatrix} \begin{bmatrix}
\vec{\pmb{\phi}}_{h}^{j}\\
\vec{\rho}_{h}^{j}\\
\end{bmatrix} 
 = \begin{bmatrix}
\SLOh\vec{\vecv{r}}_{1,h}\\
0\\
\end{bmatrix}. 
\end{equation}

Taking the Schur complement we see that $\vec{\pmb{\phi}}_{h}^{j} = \pmb{\Pi}_{j,\varepsilon, h}^{-1}\SLOh\vec{\vecv{r}}_{1,h}$, with
\begin{equation}\label{eq:schur}
\pmb{\Pi}_{j,\varepsilon, h}= \left[\mathbb{G}-B_{j}(\mathbb{N}_{\varepsilon} +\mathbb{L}\mathbb{K}_{\varepsilon}^{-1}(\mathbb{L})^{T})\right].
\end{equation}
The sum in \eqref{eq:weak_pade_sum} is now represented as
$$
\mathbb{G}\vec{\vecv{r}}_{2,h} = R_0\left(\SLOh\vec{\vecv{r}}_{1,h} - \frac{1}{R_0}\mathbb{G}\sum_{j=1}^{N_p}\frac{A_j}{B_j}\vec{\pmb{\phi}}_{h,j}\right),
$$
and the system \eqref{eq:lambda2_wf} is correspondingly solved by
$$
\vec{\vecv{r}}_{3,h} = (\mathbb{G}-\mathbb{N}_{\varepsilon})^{-1}R_0\left(\SLOh\vec{\vecv{r}}_{1,h} - \frac{1}{R_0}\mathbb{G}\sum_{j=1}^{N_p}\frac{A_j}{B_j}\vec{\pmb{\phi}}_{h,j}\right).
$$

We still have to evaluate the inverse of the isomorphism $\pmb{\Theta}$. However, this is trivial in a basis of RWG functions 
since we have that $\pmb{\Theta}(\vecv{RWG}_i) = \vecv{SNC}_i$. Hence, the isomorphism acts on the basis functions but not on the 
discrete coefficients of them. We therefore have that $\vec{\vecv{r}}_{h}=-\vec{\vecv{r}}_{3,h}$. A full preconditioned evaluation 
of the EFIE operator is therefore given as
\begin{align}\label{eq:discrete_precond}
\vec{\vecv{r}}_{h} &= -(\mathbb{G}-\mathbb{N}_{\varepsilon})^{-1}R_0\left(\mathbb{I} - \frac{1}{R_0}\mathbb{G}\sum_{j=1}^{N_p}\frac{A_j}{B_j}\pmb{\Pi}_{j,\varepsilon, h}^{-1}\right)\SLOh\vec{\vecv{r}}_{1,h},\nonumber\\
&:= \vecv{V}^{-1}_{\varepsilon,h, N_p}\SLOh\vec{\vecv{r}}_{1,h},
\end{align}
where $\mathbb{I}$ is the simple discrete identity matrix. While this preconditioner looks complicated at first, all involved operators are simple sparse matrices that are readily available in Maxwell boundary element codes or can be straight forward implemented. Moreover, the solves in the sum can be easily executed in parallel.

To better understand 
how the number of Pad\'{e} terms influences the approximation property, in Fig. \ref{fig:spectrum_prec} we plot the spectrum 
of $\vecv{V}_{\varepsilon, h}^{-1}$ (lower right plot) against that of $\vecv{\tilde{V}}_{\varepsilon, h, N_p}^{-1}$ for different values of 
the number of terms $N_p$. As $N_p$ increases we can see very nicely how the spectrum becomes very similar to the desired spectrum 
even though we do not approximate the MtE operator directly, but an approximation involving the pseudo-differential operator $(\vecv{I} + \mathcal{J})^{1/2}$.

\begin{figure}[h!]
\hspace*{-1cm}
\center
\includegraphics[width=9cm]{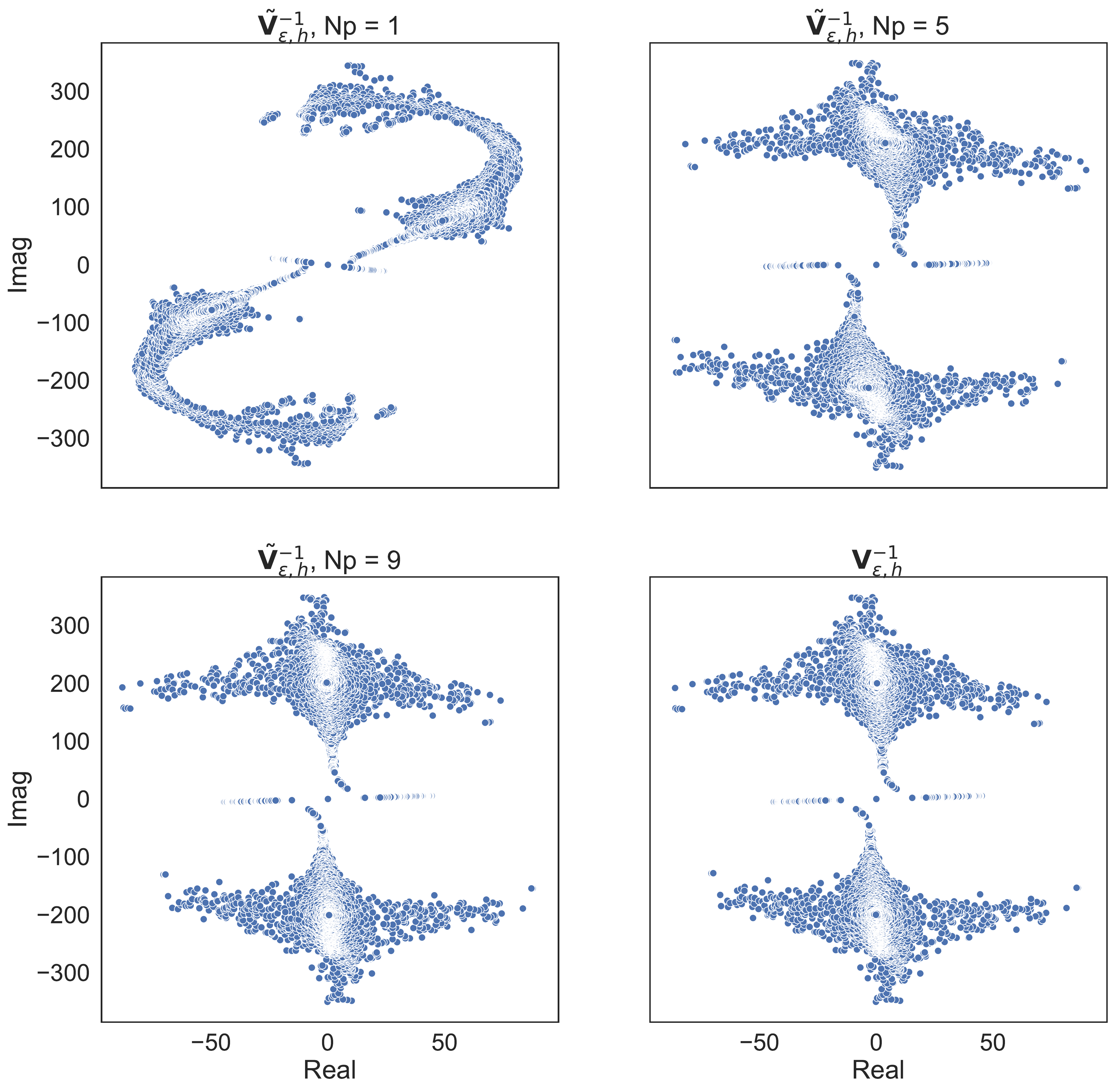}
\caption{Spectral comparison for $\vecv{V}_{\varepsilon,h}^{-1}$ and $\tilde{\vecv{V}}_{\varepsilon,h, N_{p}}^{-1}$. Top left: $N_p = 1$. Top right: $N_p = 5$. Bottom left: $N_p=9$. Bottom right $\vecv{V}_{\varepsilon, h}^{-1}$. In all figures, $\kappa=\pi$.}
\label{fig:spectrum_prec}
\end{figure}

\subsection{Implementational simplification}
\label{sec:simplification}

In this section we propose a simplification of the above preconditioner that we believe is novel and reduces its computational 
effort. In \cite{el2014approximate} the direct sparse solution of the various systems of the form \eqref{eq:matrix_system} was 
proposed. However, depending on the magnitude of the associated Pad\'e coefficients this can be simplified. Consider again the sum
\begin{equation}
\label{eq:coeff_sum}
\mathbb{I} - \frac{1}{R_0}\mathbb{G}\sum_{j=1}^{N_p}\frac{A_j}{B_j}\pmb{\Pi}_{j,\varepsilon, h}^{-1}
\end{equation}
contained in \eqref{eq:discrete_precond} with
$$
\pmb{\Pi}_{j,\varepsilon, h}= \left[\mathbb{G}-B_{j}(\mathbb{N}_{\varepsilon} +\mathbb{L}\mathbb{K}_{\varepsilon}^{-1}(\mathbb{L})^{T})\right].
$$
Let $\beta_j :=A_j/B_j$. In Fig. \ref{fig:pade_coef} the values of $\beta_j$ and $B_j$ are shown for varying $j$. We can see 
heuristically that if $\beta_j$ is large, then $B_j$ becomes small and when $\beta_j$ is small, $B_j$ remains bounded. Let us consider 
those two cases.

If $\beta_j$ is small then we can just discard the corresponding term in \eqref{eq:coeff_sum}. However, if $\beta_j$ is large then 
$B_j$ becomes negligible, and we can apply the simplification:
$$
\pmb{\Pi}_{j,\varepsilon, h}= \left[\mathbb{G}-B_{j}(\mathbb{N}_{\varepsilon} +\mathbb{L}\mathbb{K}_{\varepsilon}^{-1}(\mathbb{L})^{T})\right] \approx \mathbb{G}.
$$
We hence obtain that
\begin{equation}\label{eq:approx_coeffs}
\mathbb{I} - \frac{1}{R_0}\mathbb{G}\sum_{j=1}^{N_p}\frac{A_j}{B_j}\pmb{\Pi}_{j,\varepsilon, h}^{-1}\approx \left(1 - \frac{1}{R_0}\sum_{j\in\mathcal{I}}\frac{A_j}{B_j}\right)\mathbb{I},
\end{equation}
where $\mathcal{I}$ is the set of Pad\'{e} coefficients assumed to be dominant and retained. For all those we perform the 
approximation $\pmb{\Pi}_{j,\varepsilon, h}\approx \mathbb{G}$. Having this simplification, then \eqref{eq:discrete_precond} turns into

\begin{equation*}
\vec{\vecv{r}}_{h} = -K_{\varepsilon,N_p}(\mathbb{G}-\mathbb{N}_{\varepsilon})^{-1}\SLOh\vec{\vecv{r}}_{1,h},
\end{equation*}

where $K_{\varepsilon,N_p}$ is a constant. Since we apply the preconditioner on both sides of the equation (as in \eqref{eq:prec_efie}), we can 
dismiss the constants and just keep $(\mathbb{G}-\mathbb{N}_{\varepsilon})^{-1}$ as a preconditioner. This is not a direct approximation of the 
MtE map anymore, but as we will see it still performs well as preconditioner.

\begin{figure}[h]
    \centering
    \includegraphics[width=9cm]{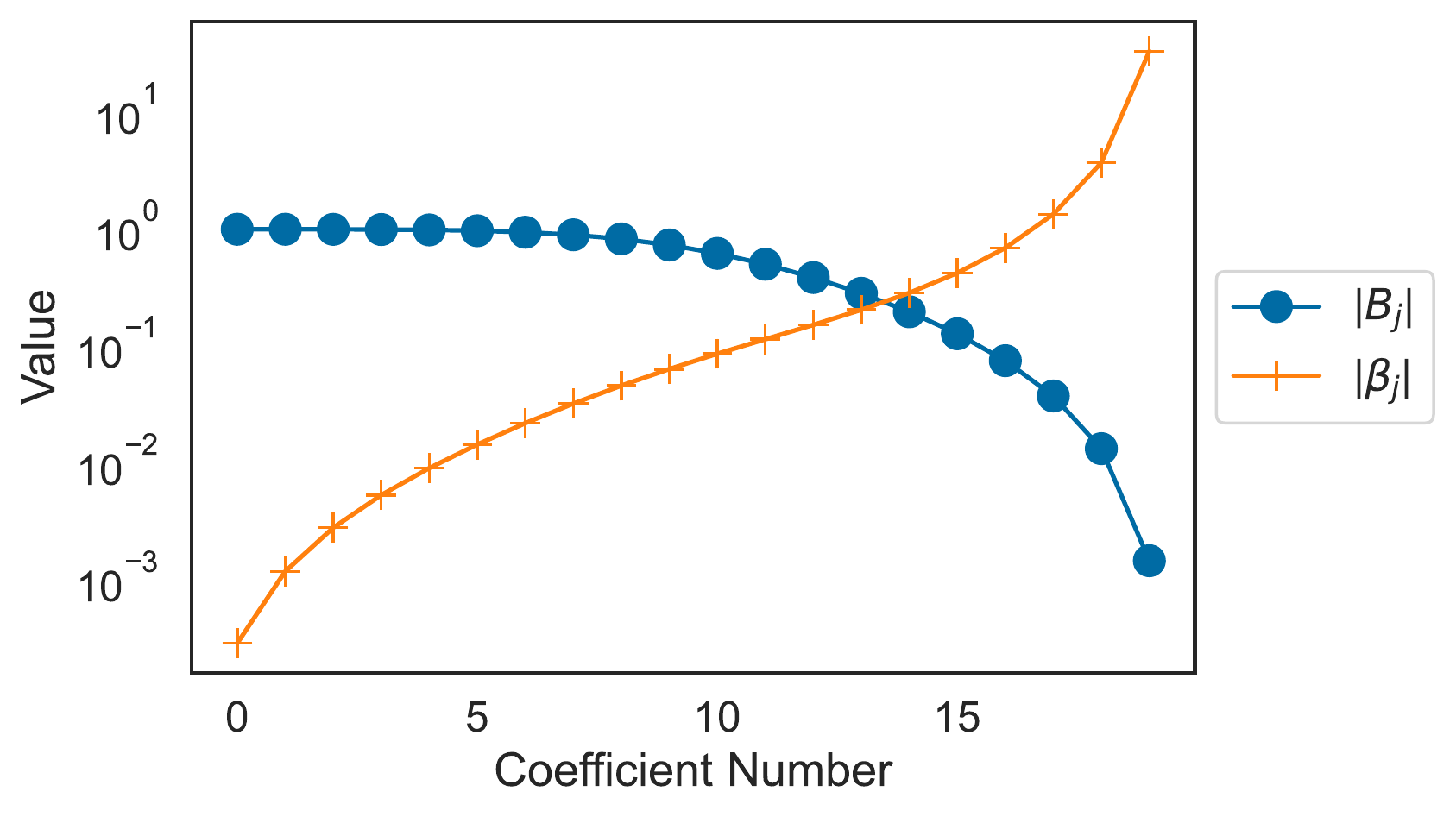}
    \caption{Pad\'e Coefficients $\beta_{j}$ and $B_{j}$ for $N_{p} = 50$.}
    \label{fig:pade_coef}
\end{figure}

\section{Numerical Experiments}\label{sect:NExp}

This section demonstrates the performance of the proposed OSRC preconditioner and some comparisons with other regularisers.  
All the tests were performed using Bempp software \cite{scroggs2017software} on a spherical grid of radius $r=1$. 
In this section we denote by $\tilde{\vecv{V}}_{\varepsilon, h,A,N_p}^{-1}$ the preconditioner obtained by solving the full block systems 
\eqref{eq:matrix_system} with a Pad\'{e} degree $N_p$ and by $\tilde{\vecv{V}}_{\varepsilon, h,B}^{-1}$ the simplified preconditioner  $(\mathbb{G}-\mathbb{N}_{\varepsilon})^{-1}$ as described in Section \ref{sec:simplification}.

\subsection{Validation} 
The OSRC operator was validated using the same method as in \cite{el2014approximate}, where bistatic RCS are calculated using:

\begin{enumerate}
\item Analytic solutions on the unit sphere calculated using spherical harmonics.
\item A direct formulation of the EFIE ($\vecv{V}_{h}^{-1}$).
\item An approximation of $\vecv{V}_{\varepsilon,h}^{-1}$, that we calculated by computing the square roots of the eigenvalues of the matrix that generates $\pmb{\Lambda}_{1,\varepsilon}$.
\item By applying $\tilde{\vecv{V}}_{\varepsilon, h,A,2}^{-1}$.
\end{enumerate}

The first test (Fig. \ref{fig:RCS}) replicates the results obtained in \cite{el2014approximate} and shows the bistatic RCS obtained from the scattering problem 
of an incident electromagnetic plane wave by a PEC unit sphere, for $\kappa = \pi$ and $\kappa = 8 \pi$. The analytic solution and the curve due to $\tilde{\vecv{V}}_h^{-1}$ 
agree up to plotting accuracy.
Moreover, the graphs of $\tilde{\vecv{V}}_{\varepsilon, h,A,2}^{-1}$ and $\vecv{V}_{\varepsilon, h}^{-1}$ overlap each other, and 
both approximate the qualitative behaviour of the bistatic RCS well, though not perfectly. 

\begin{figure}[h!]
\centering
\includegraphics[width=9cm]{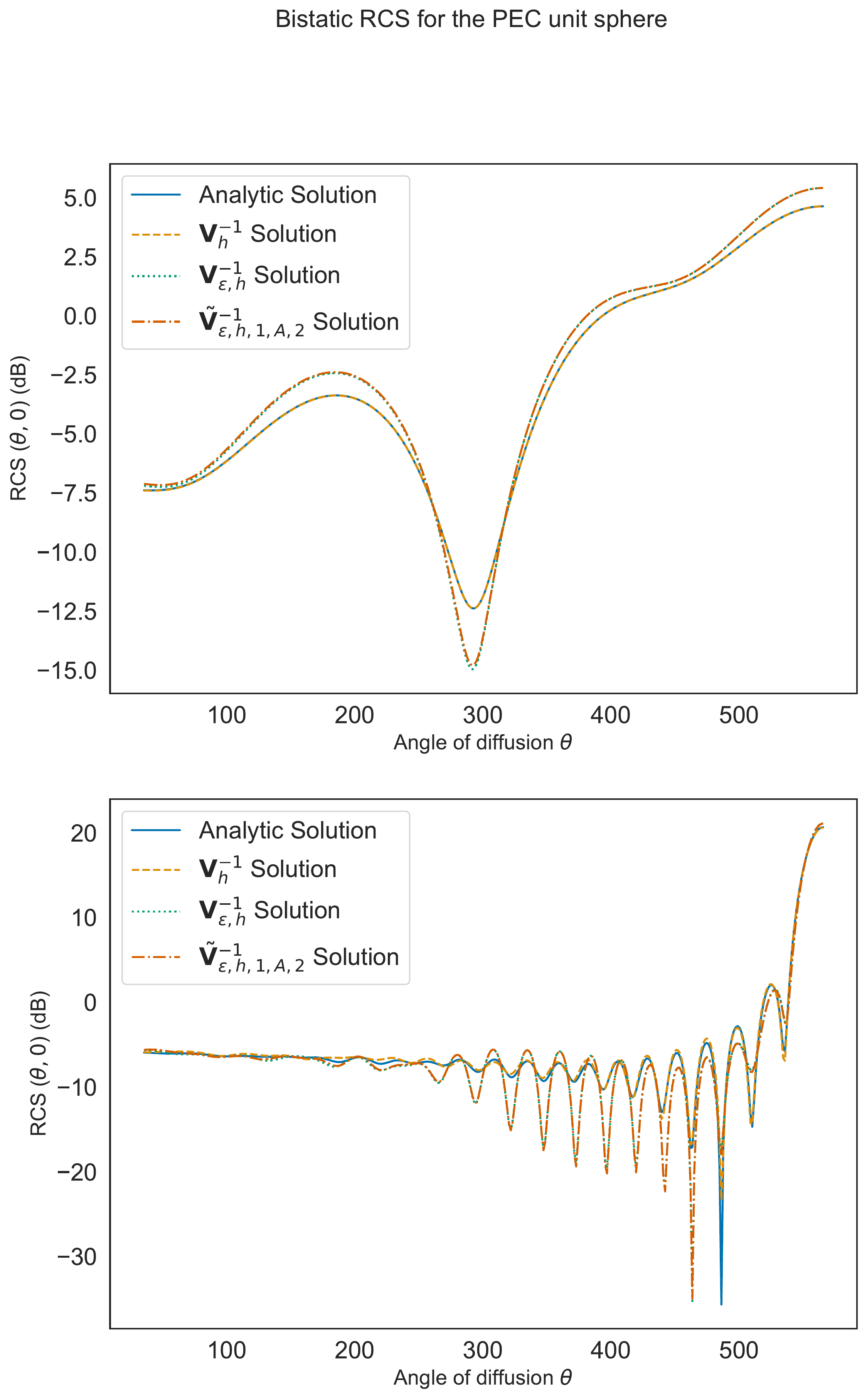}
\caption{Bistatic RCS for the PEC unit sphere illuminated by an incident electromagnetic plane wave at $\kappa = \pi$ (up) and $\kappa = 8 \pi$ (down). }
\label{fig:RCS}
\end{figure}

\subsection{Performance Comparison}

In order to compare the performance of the OSRC preconditioner to others, the following attributes were benchmarked:

\begin{enumerate}
\item[a)] GMRES number of iterations.
\item[b)] Assembly and solving times for: \begin{itemize}
\item Pure direct formulation of the EFIE, denoted by $\SLOh$ (always calculated in the primal grid).
\item EFIE regularised using the Calder\'on Multiplicative Preconditioner (CMP), denoted by $\SLOh^{2}$.
\item EFIE regularised using the Refinement Free Calder\'on Multiplicative Preconditioner (RF-CMP), denoted by RF-$\SLOh^{2}$. We must mention that this preconditioner
requires the inversion of a dense matrix. The authors in \cite{adrian2019refinement} claim that this can be solved by using a multigrid preconditioner that \textbf{we have
not implemented in this work}. Having this in mind, the assembly time recorded in this document should be larger than it could be under 
an optimisation of the method.

It is necessary to mention that in this case, the Preconditioner is designed to be applied in a discretisation of the EFIE
built using RT and NC basis functions as the standard div and curl conforming basis functions. Also, we have solved this system using the conjugate gradient 
algorithm as suggested by the authors in \cite{adrian2019refinement}.

\item EFIE regularised with $\tilde{\vecv{V}}_{\varepsilon, h, A, N_{p}}^{-1}$, with $N_{p}=1, 2$.
\item EFIE regularised with $\tilde{\vecv{V}}_{\varepsilon, h, B}^{-1}$.

\end{itemize}
\end{enumerate}

These tests were performed using $\mathcal{H}$-matrix compression \cite{hackbusch2015hierarchical} with an accuracy of around $10^{-3}$ of the compressed matrices.

As expected, in Fig. \ref{fig:its_cr}, $\tilde{\vecv{V}}_{\varepsilon, h, A, N_p}^{-1}\SLOh$ shows (slightly) better results than 
$\tilde{\vecv{V}}_{\varepsilon, h, B}^{-1}\SLOh$, because the first is a better approximation than the latter. However, in terms
of solving time (table\ref{tab:st_cr}), it could be more convenient to use $\tilde{\vecv{V}}_{\varepsilon, h, B}^{-1}\SLOh$, since it is considerably cheaper to apply. 

Table \ref{tab:at_cr} shows assembly times of the different preconditioners in relationship to the assembly time of the non-preconditioned EFIE. The best result for each column is shown in bold.
We note that the $\tilde{\vecv{V}}_{\varepsilon, h, B}^{-1}$ variant is almost as cheap to assemble as the non-preconditioned system despite still 
being excellent as preconditioner. We notice our non-optimal implementation of the RF-$\SLOh^{2}$ variant. The CMP $\SLOh^{2}$ suffers in terms of 
assembly time from the need to assemble operators on the barycentrically refined grid.
The solution times in Table \ref{tab:st_cr} also confirm the effectiveness of the OSRC type preconditioners and here in particular the $\tilde{\vecv{V}}_{\varepsilon, h, B}^{-1}$ variant.

\begin{figure}[htbp]
\centering
\includegraphics[width=9.6cm]{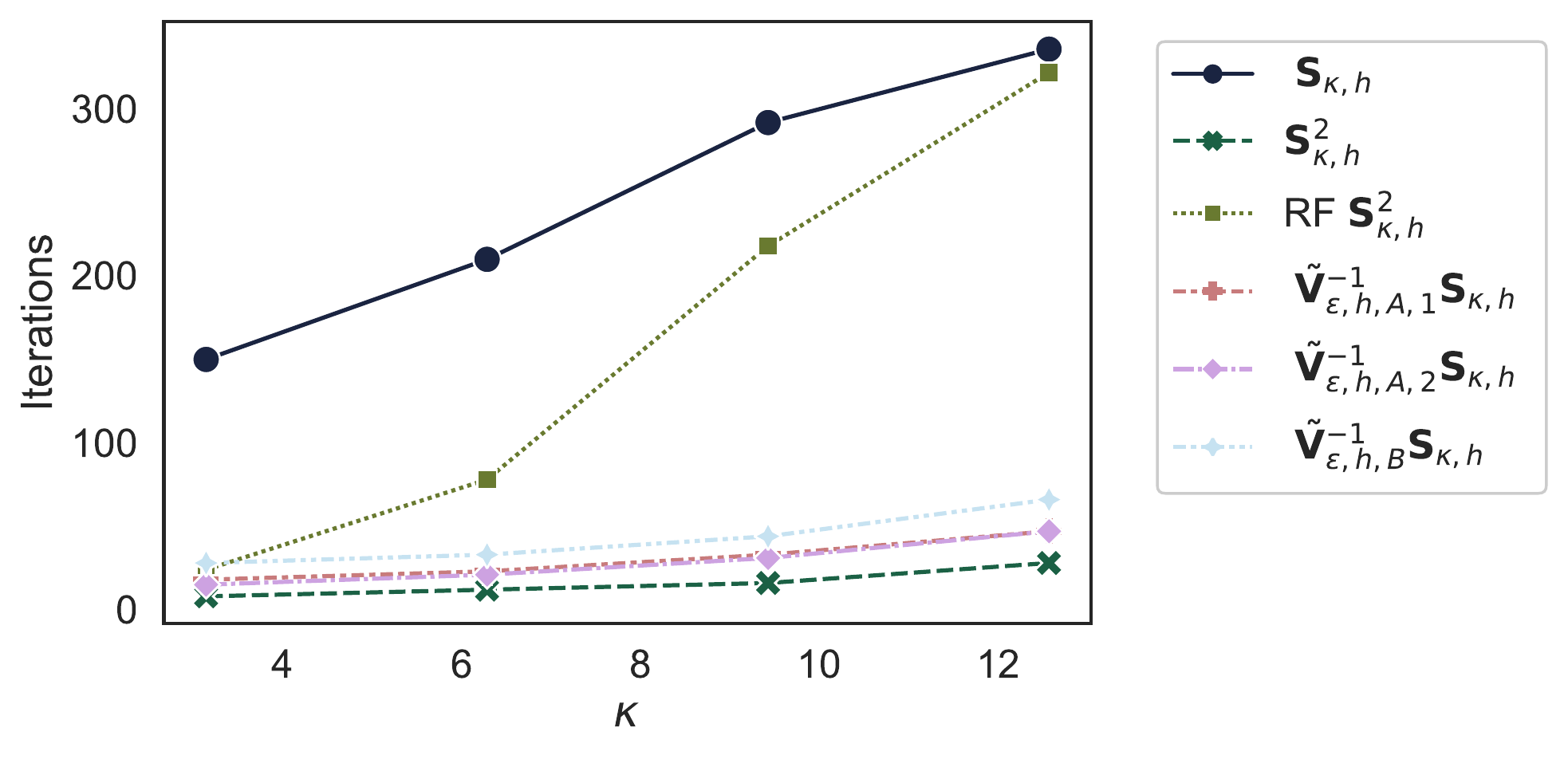}
\caption{Iterations comparison between different EFIE formulations on a grid with constant relation $\kappa \cdot h$.}
\label{fig:its_cr}
\end{figure}

\begin{table}[h!]
\begin{center}
\caption{\label{tab:at_cr}$\vecv{T}(\vecv{R}\SLOh) \mathbin{/} \vecv{T}(\SLOh)$ assembly time ratios comparison between different EFIE formulations on a spherical grid with constant relation $\kappa \cdot h$.}
\begin{tabular}{|c|c|c|c|c|c|}
\hline
 Formulation & $\kappa= \pi$  &   $\kappa= 2\pi$ &   $\kappa= 3\pi$ &    $\kappa= 4\pi$  \\ \hline              
$\SLOh$       &    1.000 &   1.000 &   1.000 &   1.000 \\ \hline
$\SLOh^{2}$       &  10.043 &   8.929 &  13.330 &   11.173  \\ \hline
RF-$\SLOh^{2}$       &  5.644 &  23.132 &  67.077 &  132.672 \\ \hline
 $\tilde{\vecv{V}}_{\varepsilon, h, A,1}^{-1}\SLOh$ &  1.063 &   1.141 &   1.203 &    1.234\\ \hline
 $\tilde{\vecv{V}}_{\varepsilon, h, A,2}^{-1}\SLOh$    &  1.129 &   1.233 &   1.365 &    1.434  \\ \hline
 $\tilde{\vecv{V}}_{\varepsilon, h, B}^{-1}\SLOh$   & \textbf{1.009} &   \textbf{1.014} &   \textbf{1.015} &    \textbf{1.015}  \\ \hline
\end{tabular}
\end{center}
\end{table}

\begin{table}[h!]
\begin{center}
\caption{\label{tab:st_cr}$\vecv{T}(\vecv{R}\SLOh) \mathbin{/} \vecv{T}(\SLOh)$ solving time ratios comparison between different EFIE formulations on a spherical grid with constant relation $\kappa \cdot h$. }
\begin{tabular}{|c|c|c|c|c|c|}
\hline
 Formulation   & $\kappa= \pi$  &   $\kappa= 2\pi$ &   $\kappa= 3\pi$ &    $\kappa= 4\pi$  \\ \hline               
$\SLOh$       &     1.000 &  1.000 &  1.000 &   1.000 \\ \hline
$\SLOh^{2}$       &    2.168 &   1.762 &   1.521 &   2.279\\ \hline
RF-$\SLOh^{2}$       &    2.938 &  17.189 &  41.086 &  80.886\\ \hline
 $\tilde{\vecv{V}}_{\varepsilon, h, A,1}^{-1}\SLOh$ &  0.319 &   \textbf{0.293} &   0.211 &   0.253\\ \hline
 $\tilde{\vecv{V}}_{\varepsilon, h, A,2}^{-1}\SLOh$ & 0.575 &   0.402 &   0.263 &   0.351 \\ \hline
 $\tilde{\vecv{V}}_{\varepsilon, h, B}^{-1}\SLOh$ &  \textbf{0.187} &   0.309 &   \textbf{0.121} &   \textbf{0.200} \\ \hline
\end{tabular}
\end{center}
\end{table}

\subsection{$\vecv{V}_{\varepsilon,h}^{-1}$ performance under mesh refinement.}

\subsubsection{High frequency regime}

Figure \ref{fig:its_vd_hf} compares the various preconditioners for fixed wavenumber $\kappa=\pi$ and decreasing mesh width. While the EFIE without preconditioning suffers from the well known ill-conditioning problems all
other preconditioners keep the number of iteration bounded or only very slowly increasing.

\begin{figure}[htbp]
\centering
\includegraphics[width=9.5cm]{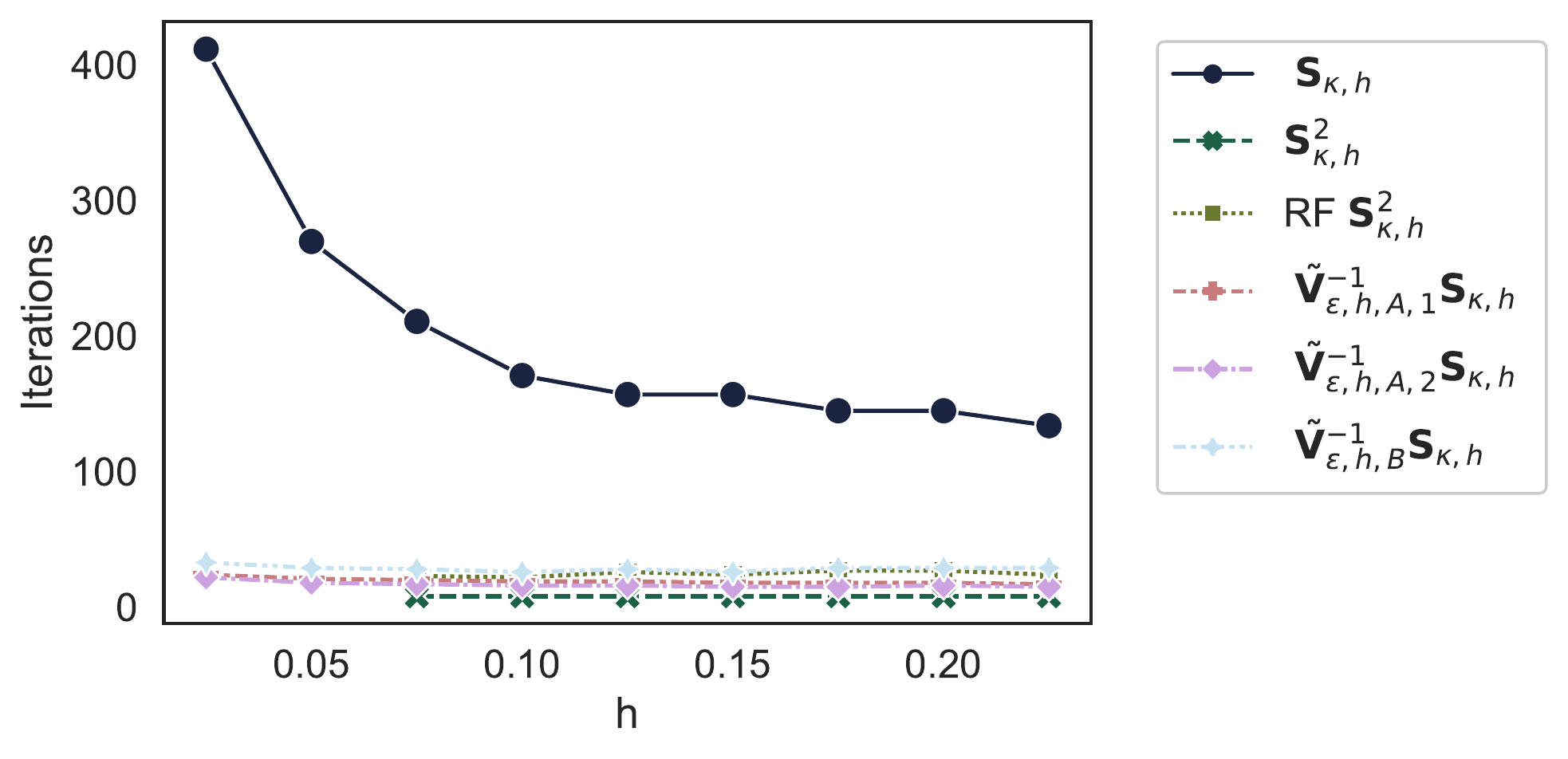}
\caption{Iterations comparison between different EFIE formulations on a grid with varying discretisation, in the high frequency regime.}
\label{fig:its_vd_hf}
\end{figure}

\subsubsection{Low frequency regime}

In this scenario ($\kappa = \pi/10$, Fig. \ref{fig:its_vd_lf}), the regularised systems keep showing a robust behaviour. However, the main difference we observe
is that the CMP and RF-CMP perform better (in terms of iterations) than the MtE preconditioner. This can be explained by the fact that 
the latter is based on a high frequency approximation, whereas, the RF-CMP is based on a low frequency approximation.

\begin{figure}[htbp]
\centering
\includegraphics[width=9.5cm]{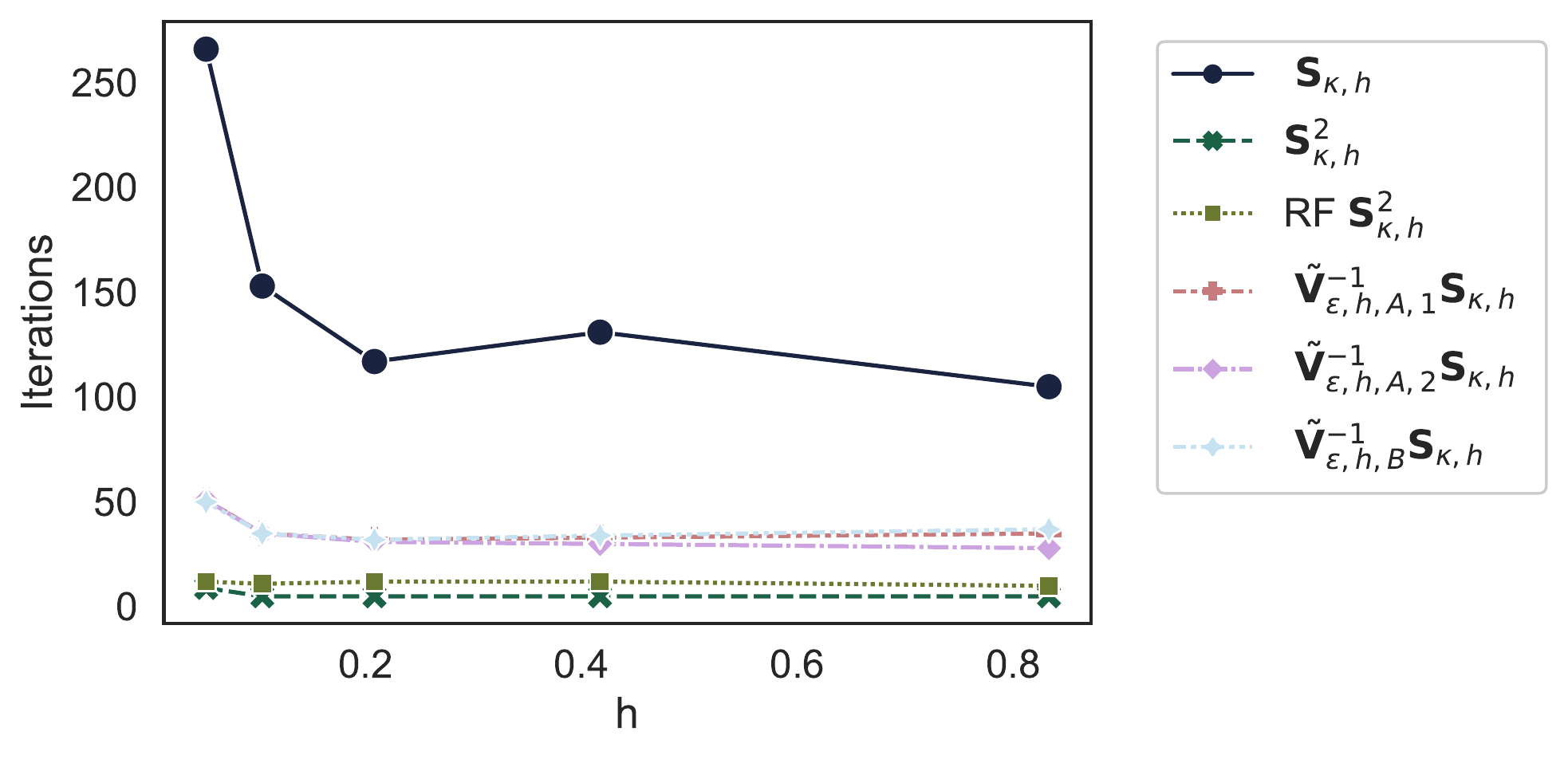}
\caption{Iterations comparison between different EFIE formulations on a grid with varying discretisation, in the low frequency regime.}
\label{fig:its_vd_lf}
\end{figure}

\subsubsection{Time performance}

In terms of time, table \ref{tab:at_vd} shows that in general,  assembling $\tilde{\vecv{V}}_{\varepsilon, h, *,N_p}^{-1}\SLOh$ remains cheap as $h\rightarrow 0$. 
Table \ref{tab:st_vd} shows that the solving time ratios of $\tilde{\vecv{V}}_{\varepsilon, h, *,N_p}^{-1}\SLOh$ improve for smaller $h$ as the number of iterations 
that takes to solve the problem remains approximately constant. In both regards, assembly and solving time, $\tilde{\vecv{V}}_{\varepsilon, h, *,N_p}^{-1}\SLOh$ 
outperforms $\SLOh^2$ (unless the discretisation is very rough).

\begin{table}[h!]
\begin{center}
\caption{\label{tab:at_vd}$\vecv{T}(\vecv{R}\SLOh) \mathbin{/} \vecv{T}(\SLOh)$ assembly time ratios comparison between different EFIE formulations on a spherical grid with varying discretisation.}
\begin{tabular}{|c|c|c|c|c|c|c|}
\hline
$\kappa$ & \multicolumn{2}{c|}{$\pi/10$} & \multicolumn{2}{c|}{$\pi$} \\ \hline
Formulation   &   $h=0.052$ &    $h=0.833$&    $h=0.075$&    $h=0.225$  \\ \hline
$\SLOh$       &   1.000 &   1.000 &   1.000 &   1.000   \\ \hline
$\SLOh^2$       &  12.589 &  9.565 &  11.645 &  10.450  \\ \hline
RF-$\SLOh^2$       &  85.990 &   2.446 &    30.652 &   3.568  \\ \hline
$\tilde{\vecv{V}}_{\varepsilon, h, A,1}^{-1}\SLOh$   &  1.179 &   1.104 &   1.122 &   1.051 \\ \hline
$\tilde{\vecv{V}}_{\varepsilon, h, A,2}^{-1}\SLOh$ &   1.221 &   1.195 &   1.257 &   1.077  \\ \hline
$\tilde{\vecv{V}}_{\varepsilon, h, B}^{-1}\SLOh$  &  \textbf{1.019} &   \textbf{1.049} &   \textbf{1.015} &   \textbf{1.012} \\ \hline
\end{tabular}
\end{center}
\end{table}

\begin{table}[h!]
\begin{center}
\caption{\label{tab:st_vd}$\vecv{T}(\vecv{R}\SLOh) \mathbin{/} \vecv{T}(\SLOh)$ solving time ratios comparison between different EFIE formulations on a spherical grid with varying discretisation.}
\begin{tabular}{|c|c|c|c|c|c|c|}
\hline
$\kappa$ & \multicolumn{2}{c|}{$\pi/10$} & \multicolumn{2}{c|}{$\pi$} \\ \hline
Formulation   &   $h=0.052$ &    $h=0.833$&    $h=0.075$&    $h=0.225$  \\ \hline
$\SLOh$       &   1.000 &   1.000 &   1.000 &   1.000   \\ \hline
$\SLOh^2$       &  1.629 &  1.493 &  1.584 &  2.489  \\ \hline
RF-$\SLOh^2$       &  4.137 &   \textbf{0.783} &    6.040 &   1.636  \\ \hline
$\tilde{\vecv{V}}_{\varepsilon, h, A,1}^{-1}\SLOh$   &  0.422 &   1.104 &   0.222 &   0.338 \\ \hline
$\tilde{\vecv{V}}_{\varepsilon, h, A,2}^{-1}\SLOh$ &   0.552 &   1.006 &   0.316 &   0.457  \\ \hline
$\tilde{\vecv{V}}_{\varepsilon, h, B}^{-1}\SLOh$  &  \textbf{0.180} &   1.430 &   \textbf{0.184} &   \textbf{0.231} \\ \hline
\end{tabular}
\end{center}
\end{table}

\subsection{An example on a less regular surface.}

The same tests were performed on a NASA almond-shaped grid, with a wavenumber of $\kappa = 2\pi$, to see if the condition number boundedness was preserved on 
a more interesting shape. Figure \ref{fig:it_vd_almond} shows that, compared to the non-regularised EFIE, the MtE preconditioner 
allows to solve the problem in considerably less iterations than the non-regularised EFIE, while being substantially cheaper to 
compute (compared to the CMP preconditioners). 

\begin{figure}[htbp]
\centering
\includegraphics[width=9.7cm]{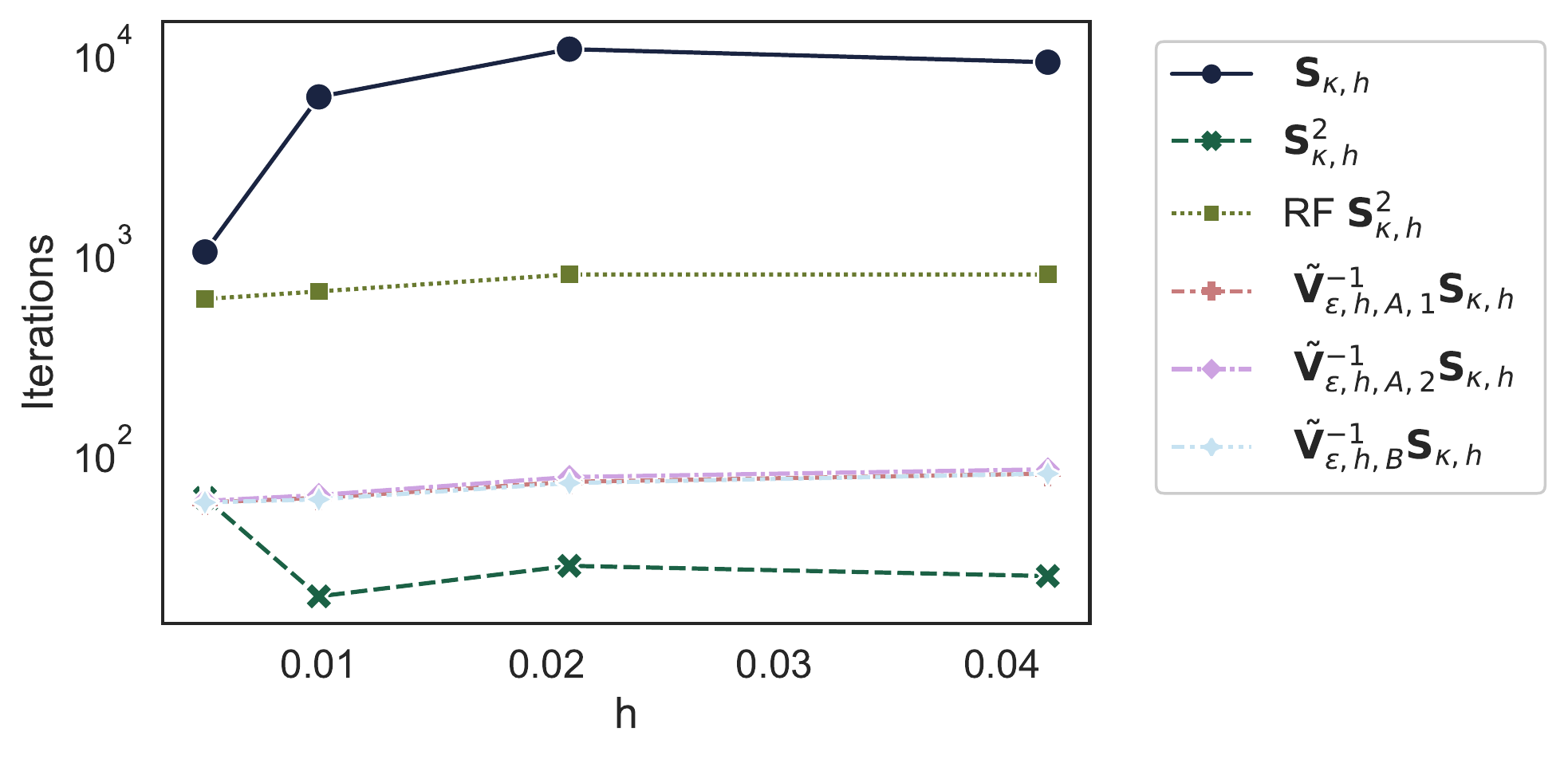}
\caption{Iterations comparison between different EFIE formulations on an almond-shaped grid with varying $h$.}
\label{fig:it_vd_almond}
\end{figure}

\begin{figure}[htbp]
\hspace*{-0.2cm}  
\centering
\includegraphics[width=9.5cm]{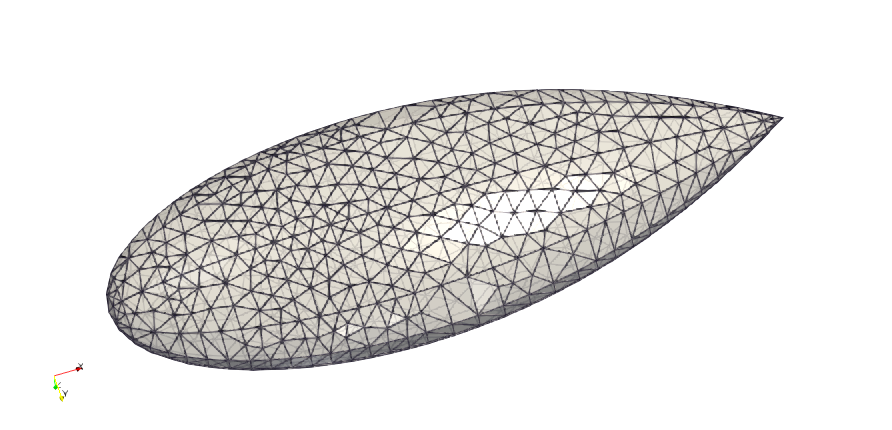}
\caption{Example of a NASA almond grid.}
\label{fig:nasa_almond}
\end{figure}

\subsection{Heads up: the MtE on an open surface}
Finally, in Fig. \ref{fig:its_cylinder} we demonstrate the performance of the MtE preconditioner on an open cylinder 
(see Fig. \ref{fig:cylinder} for the geometry). The formulation is the usual EFIE for screens \cite[section 7.1]{Buffa2003}. As Fig. 
\ref{fig:its_cylinder} shows, the iteration count for the MtE preconditioned version behaves very favourably even compared 
to the CMP adapted to screens.

\begin{figure}[htbp]
\centering
\includegraphics[width=9.7cm]{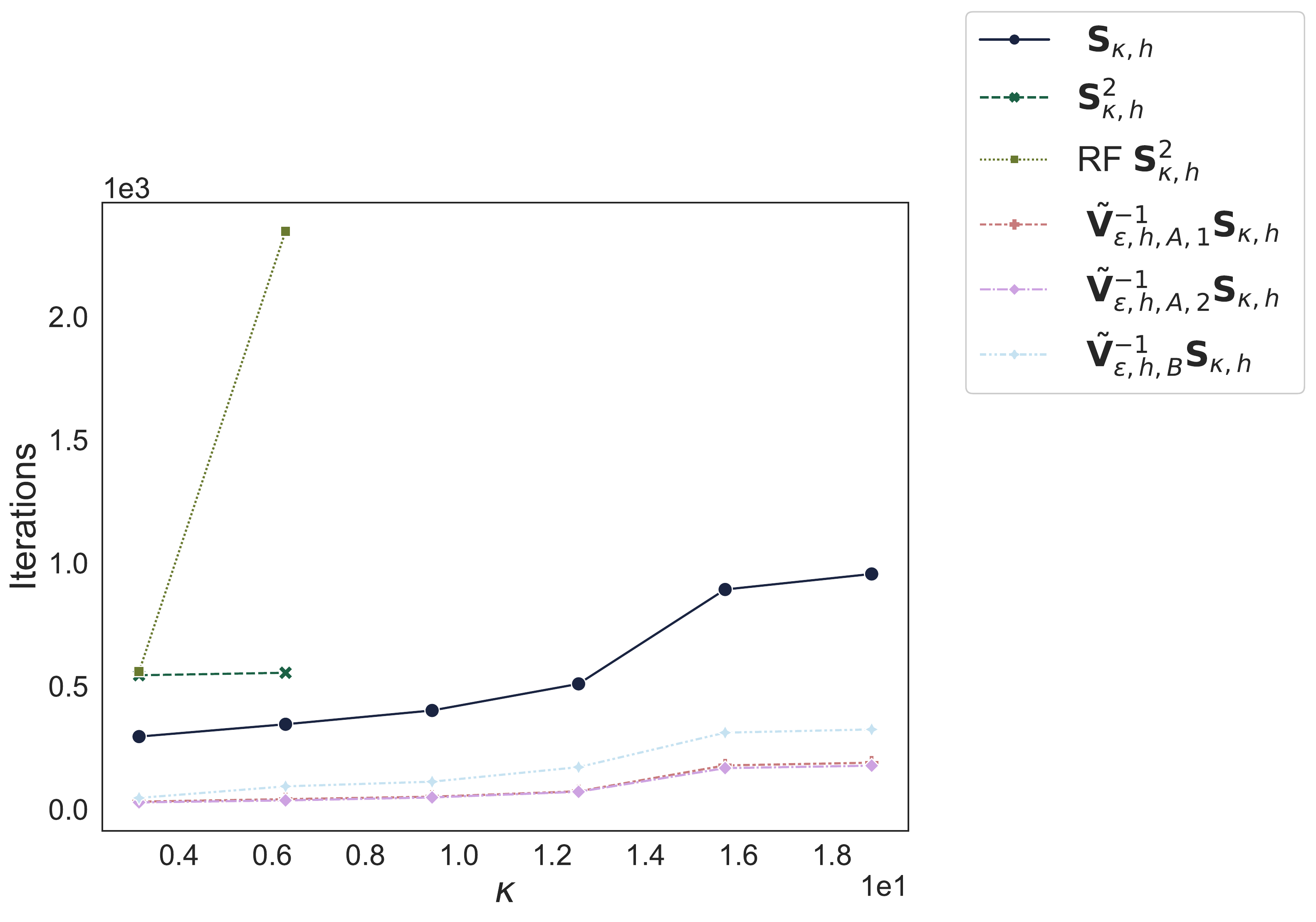}
\caption{Iterations comparison between different EFIE formulations on an open cylinder grid, with varying $\kappa$.}
\label{fig:its_cylinder}
\end{figure}

\begin{figure}[htbp]
    \hspace*{-0.2cm}  
    \centering
    \includegraphics[width=9cm]{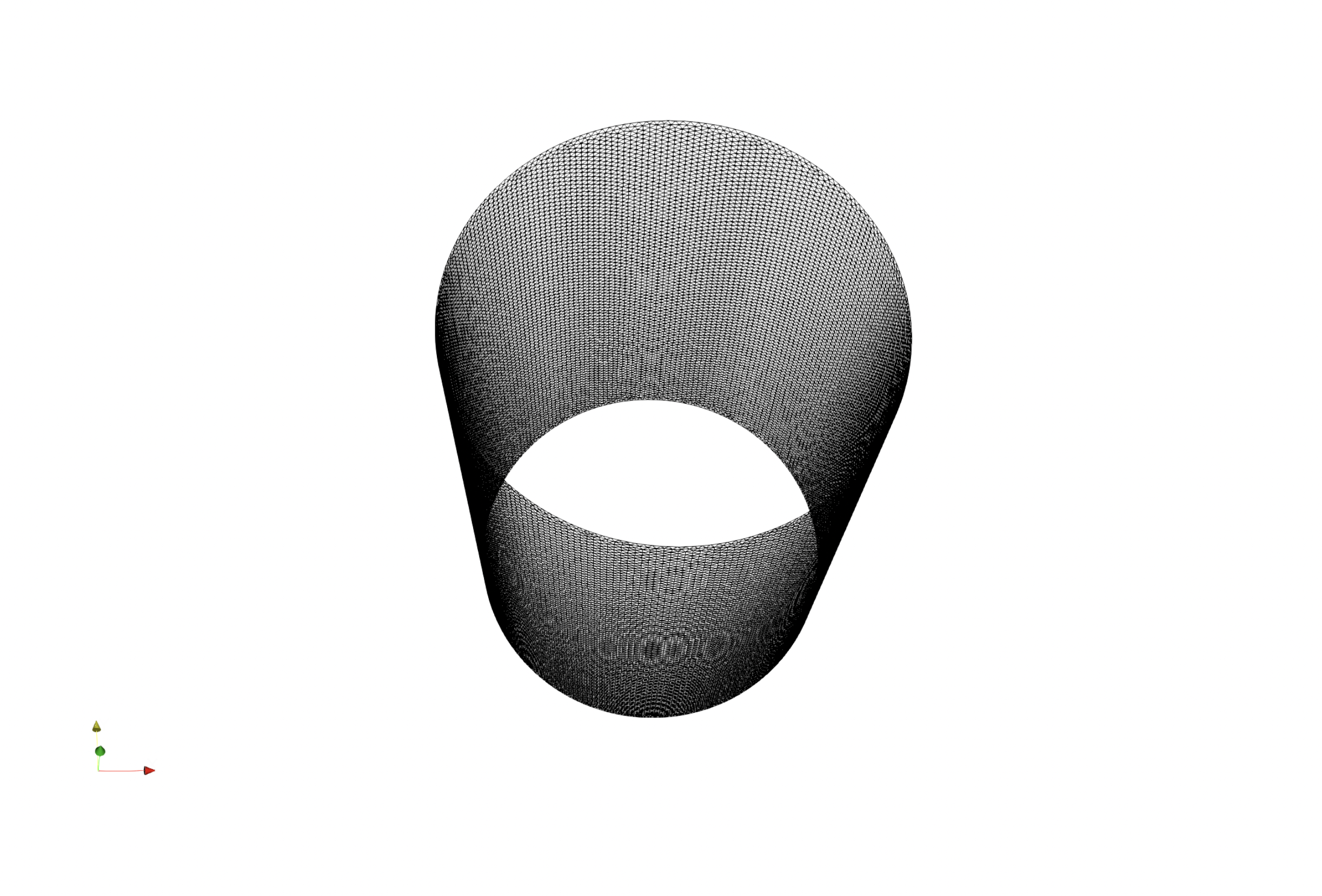}
    \caption{Example of open cylinder grid.}
    \label{fig:cylinder}
    \end{figure} 

\section{Conclusions}\label{sect:CR}

The aim of this paper was to test an approximation of the Magnetic to Electric operator proposed in \cite{el2014approximate} as a preconditioner operator for the EFIE. This operator ($\tilde{\vecv{V}}_{\varepsilon, N_p}^{-1}$) is based on a rational complex Padé approximant of an OSRC ($\vecv{V}_{\varepsilon}^{-1}$) operator, also proposed in \cite{el2014approximate}.  

It was shown that this operator works as a preconditioner for the EFIE and different alternatives for its discretisation 
were proposed and benchmarked. The results from these tests prove the effectiveness of the proposed preconditioner and that 
it also outperforms the standard Calderón Multiplicative Preconditioner. It is also competitive with a refinement free 
version of the CMP, while usually being simpler to implement as it uses just straight forward sparse matrix discretisations 
of surface differential operators that are often already available or easy to implement within boundary element codes.
On top of this, we performed tests on more complex geometries, including an open domain. The results obtained from these tests are
very promising, opening the door to the application of this method to more complex open surfaces.

\bibliography{mybibfile}{}
\bibliographystyle{plain}

\end{document}